\documentclass[11 pt]{amsart}
\usepackage{graphicx} 
\usepackage{amsmath}
\usepackage{mathtools}
\usepackage
{amstext, amscd, setspace, tikz-cd,mathtools, enumerate, graphics, latexsym, siunitx}

\usepackage{pst-node}
\usepackage{tikz-cd} 

\usepackage{amssymb}
\usepackage{hyperref}

\usepackage{PTSerif} 
\usepackage{graphicx}

\usepackage[cmtip,all]{xy}
\newcommand{\properideal}{%
\mathrel{\ooalign{$\lneq$\cr\raise.22ex\hbox{$\lhd$}\cr}}}

\newcommand{\Z}{{\mathbb Z}}

\theoremstyle{definition}
\numberwithin{equation}{subsubsection}

\newtheorem{theorem}{Theorem}[section]
\newtheorem{corollary}{Corollary}[theorem]

\newtheorem{lemma}[theorem]{Lemma}
\newtheorem{proposition}[theorem]{Proposition}

\newtheorem{definition}[theorem]{Definition}
\newtheorem{example}[theorem]{Example}

\newtheorem{remark}{Remark}[theorem]

\numberwithin{equation}{section}

\title[Root Clusters and Multiclusters over Imperfect Hilbertian Fields]{Root Clusters and Multiclusters over Imperfect Hilbertian Fields}
\author{Shubham Jaiswal}

\address{Department of Mathematics, IIT Bombay, Powai, Mumbai 400 076, India.}

\email{mynameissj555@gmail.com}


\date{\today}


\begin{document}

\begin{abstract}
    We extend the theory of root clusters from perfect fields to general fields which are not necessarily perfect. We introduce the following notions for field extensions over any given base field and study their interesting properties: root cluster size, multicluster size and their generalizations root capacity, multiroot capacity; ascending index, ascending normal index and their generalizations intersection indicium, intersection normal indicium; compositum indicium and compositum normal indicium. We establish our results on the Inverse problems
for these generalized notions over Hilbertian fields which generalizes our earlier results which were over number fields. In particular, we show over a given Hilbertian field, the existence of a polynomial for given degree, cluster size and multicluster size and existence of an extension for given root capacity and multiroot capacity with respect to that polynomial.
\end{abstract}

\maketitle

\section{Introduction}

The theory of root clusters over perfect fields (in particular, over number fields) was substantially developed by the author and others in \cite{Bhagwat_2025}, \cite{jaiswal2025rootcapacityintersectionindicium}, \cite{jaiswal2025variantsinverseclustersize}, \cite{krithika2023root}, \cite{perlis2004roots}. In this paper we extend the theory of root clusters to general fields which are not necessarily perfect. In Problem 10.2.4 in Chapter 10 of PhD Thesis \cite{ShubhamJaiswalPhDThesis} of the author, it was asked that can we generalize the Inverse cluster size problem for number fields (Theorem 3.1.1 in \cite{Bhagwat_2025}) and the Inverse ascending index problem for number fields (Theorem 9.0.5 in \cite{Bhagwat_2025}) from number
fields to a bigger class of fields? These and more general questions are answered in this paper. 

\smallskip

In Section \ref{notion section}, we introduce the following notions for field extensions over any given base field: root cluster size, multicluster size, descending normal index, ascending index and ascending normal index. We establish interesting properties of these quantities like the following : The cluster size of an extension divides the cluster size of the maximal separable subextension of the extension (Proposition \ref{r prop} (4)) and these two quantities need not be equal (Example \ref{interesting eg}) and the quotient of the cluster size of the maximal separable subextension divided by the cluster size of the extension is the number of distinct extensions isomorphic to a given extension and having the same maximal separable subextension  (Proposition \ref{s/s'}); The ascending normal index of an extension is the product of ascending index of the extension and the degree of the maximal purely inseparable subextension of the extension (Proposition \ref{ad prop} (1)); The multicluster size of an extension is equal to the descending normal index of the extension (Proposition \ref{unique desc normal} (3)). \smallskip

 We enrich the theory by means of systematic generalizations of the above notions to more expansive frameworks of root capacity, multiroot capacity, intersection indicium, intersection normal indicium, compositum indicium and compositum normal indicium that provide metrics that allow
for a more nuanced measurement of how isomorphic copies of a given extension are distributed within various field extensions. We prove important properties of these quantities. We also give many equivalent criteria for an extension to be normal in terms of the above quantities.

\smallskip

In Section \ref{GNM Section}, we introduce the notion of general normal magnification for any finite extension as a generalization of general magnification which was for finite separable extensions. We establish how the above quantities behave under general normal magnification in Propositions \ref{GNM theorem} and \ref{GNM thm II}. We introduce the notions of unique chains and unique normal chains for extensions (which naturally arise as a consequence of cluster size, ascending index, ascending normal index and descending normal index) and study their properties and observe how they behave under general normal magnification.\smallskip

In Section \ref{inv prob section}, we give proofs of all the following Inverse problems over Hilbertian fields.

\subsection{Inverse problem on Cluster Size, Multicluster Size, Root Capacity and Multiroot Capacity over Hilbertian fields}
\hfill

The following is Theorem 1.3 in \cite{jaiswal2025rootcapacityintersectionindicium} which is referred to as the Inverse root capacity problem for number fields and which also implies Theorem 3.1.1 in \cite{Bhagwat_2025} referred to as the Inverse cluster size problem for number fields.

\begin{theorem} \label{inv root cap} Let $K$ be a number field. Given $(n,r,\rho)$ where $n>2$ and $r\mid n$ and $r\mid \rho$ and $\rho\leq n$ and $\rho\neq n-1$. There exist extensions $L/K$ and $M/K$ such that $[L:K]=n$ and cluster size $r_K(L)=r$ and root capacity $\rho_K (M,L)=\rho$. For $\rho\neq 0$, we get $M/K$ as an extension of $L/K$ contained in $\tilde{L}$.
    
\end{theorem}

We extend Theorem \ref{inv root cap} for Hilbertian fields.

\begin{theorem}\label{root cap gen res}
Let $K$ be a Hilbertian field. Suppose $K$ is either perfect of any characteristic or $K$ is imperfect with $char(K)=p>0$. Given $(n,r,l,\rho,\lambda)$ where $n>1$ and $l\mid n$ and $l\mid\lambda$ and $\lambda\leq n$ and $l/r=\lambda/\rho=p^{\mu}$ where $\mu\geq 0$ (for $K$ perfect, $\mu=0$). Assume that in the case when $r=1$ we have $\lambda\neq n-p^{\mu}$ and $n\neq 2p^{\mu}$.\smallskip

Then there exist extensions $L/K$ and $M/K$ such that $[L:K]=n$ and cluster size $r_K(L)=r$ and multicluster size $l_K(L)=l$ and root capacity $\rho_K (M,L)=\rho$ and multiroot capacity $\lambda_K(M,L)=\lambda$. For $\lambda \neq 0$, we get $M/K$ as an extension of $L/K$ contained in $\tilde{L}$.\smallskip

In fact we obtain $L/K$ as simple extension i.e. $L=K(\alpha)$ for some $\alpha\in \bar{K}$. Thus there exists a degree $n$ polynomial $f$ over $K$ and an extension $M/K$ with cluster size $r_K(f)=r$ and multicluster size $l_K(f)=l$ and root capacity $\rho_K(M,f)=\rho$ and multiroot capacity $\lambda_K(M,f)=\lambda$.

\end{theorem}

\subsection{Inverse problem on Ascending Index, Ascending Normal Index, Intersection Indicium and Intersection Normal Indicium over Hilbertian fields}
\hfill

The following is Theorem 1.5 in \cite{jaiswal2025rootcapacityintersectionindicium} which is referred to as the Inverse intersection indicium problem for number fields and which also implies Theorem 9.0.5 in \cite{Bhagwat_2025} referred to as the Inverse ascending index problem for number fields.

\begin{theorem}\label{inv tau}
    Let $K$ be a number field. Given $(n,t,\tau)$ where $n>2$ and $t\ |\ \tau\ |\ n$. Assume that in the case when $t=1$ we have $\tau\neq 2$ and $n\neq 2\tau$ and assume that in the case when $t$ is odd we have either $\tau\neq 2t$ or $n\neq 2\tau$. There exist extensions $L/K$ and $M/K$ such that $[L:K]=n$ and ascending index $t_K(L)=t$ and intersection indicium $\tau_K (M,L)=\tau$. We get $M/K$ as an extension of $L/K$ contained in $\tilde{L}$.

\end{theorem}
   
We extend Theorem \ref{inv tau} for Hilbertian fields.

\begin{theorem}\label{tau gen res}
Let $K$ be a Hilbertian field. Suppose $K$ is either perfect of any characteristic or $K$ is imperfect with $char(K)=p>0$. Given $(n,t,a,\tau,\alpha)$ where $n>1$ and $a\mid \alpha \mid n$ and $a/t=\alpha/\tau=p^{\mu}$ where $\mu\geq 0$ (for $K$ perfect, $\mu=0$). Assume that in the case when $t=1$ we have $\alpha\neq 2p^{\mu}$ and $n\neq 2\alpha$ and $n\neq 2p^{\mu}$ and assume that in the case when $t$ is odd we have either $\alpha\neq 2a$ or $n\neq 2\alpha$. \smallskip

Then there exist extensions $L/K$ and $M/K$ such that $[L:K]=n$ and ascending index $t_K(L)=t$ and ascending normal index $a_K(L)=a$ and intersection indicium $\tau_K (M,L)=\tau$ and intersection normal indicium $\alpha_K(M,L)=\alpha$. For $\alpha \neq 0$, we get $M/K$ as an extension of $L/K$ contained in $\tilde{L}$. In fact we obtain $L/K$ as simple extension i.e. $L=K(\alpha)$ for some $\alpha\in \bar{K}$.
\end{theorem}

\subsection{Inverse problem on Compositum Indicium and Compositum Normal Indicium over Hilbertian fields}
\hfill

The following is Theorem 1.6 in \cite{jaiswal2025rootcapacityintersectionindicium} which encapsulates some interesting cases of the Inverse compositum indicium problem for number fields.

\begin{theorem}\label{inv gamma}

Let $K$ be a number field. Given $(n,\gamma)$ where $n>2$ and $n \mid \gamma\mid  n!$. Assume that $n=2^m a_1a_2\cdots a_k$ with each $a_i>2$ and $m=0$ or $1$ and $k\geq 1$ and $\gamma=2^m b_1b_2\cdots b_k$ with each (i) $b_i=\ ^{a_i}P_j$ for $j\leq a_i$ or (ii) $b_i=a_i\phi(a_i/l)$ for $a_i$ odd and $l\mid a_i$ (where $\phi$ is the Euler totient function) or (iii) $b_i=a_i\cdot r^{a-1}$ for $r>1$ and $r\mid a_i$ and $a\leq (a_i/r)$. There exist extensions $L/K$ and $M/K$ such that $[L:K]=n$ and compositum indicium $\gamma_K (M,L)=\gamma$. We get $M/K$ as an extension of $L/K$ contained in $\tilde{L}$.

\end{theorem}

We extend Theorem \ref{inv gamma} for Hilbertian fields.

\begin{theorem}
    \label{gamma gen res}
Let $K$ be a Hilbertian field. Suppose $K$ is either perfect of any characteristic or $K$ is imperfect with $char(K)=p>0$. Given $(n,\gamma, \Gamma)$ where $n>1$ and $n \mid \Gamma\mid  n!$ and $\Gamma/\gamma=p^{\mu}$ where $\mu\geq 0$ (for $K$ perfect, $\mu=0$). Assume that $n=2^m a_1a_2\cdots a_k \cdot p^{\mu}$ with each $a_i>2$ and $m=0$ or $1$ and $\Gamma=2^m b_1b_2\cdots b_k \cdot p^{\mu}$ with each (i) $b_i=\ ^{a_i}P_j$ for $j\leq a_i$ or (ii) $b_i=a_i\phi(a_i/l)$ for $a_i$ odd and $l\mid a_i$ or (iii) $b_i=a_i\cdot r^{a-1}$ for $r>1$ and $r\mid a_i$ and $a\leq (a_i/r)$. Assume in the case $n=2^m p^{\mu}$ that $\Gamma=n$. \smallskip

Then there exist extensions $L/K$ and $M/K$ such that $[L:K]=n$ and compositum indicium $\gamma_K (M,L)=\gamma$ and compositum normal indicium $\Gamma_K(M,L)=\Gamma$. We get $M/K$ as an extension of $L/K$ contained in $\tilde{L}$. In fact we obtain $L/K$ as simple extension i.e. $L=K(\alpha)$ for some $\alpha\in \bar{K}$. \end{theorem}


\smallskip

\section{Root Clusters and Multiclusters}\label{notion section}

Let $K$ be a field which is not necessarily perfect. We fix an algebraic closure $\bar{K}$ once and for all and work with finite extensions of $K$ contained in $\bar K$.

\subsection{Root Multicluster Size, Descending Normal Index and Ascending Normal Index}\hfill

Let $f \in K[t]$ be an irreducible polynomial of degree $n$. We denote the set of all distinct roots of $f$ by $R$ and the multiset of all roots of $f$ counted with multiplicity by $R'$. The following notion was introduced by Perlis in \cite{perlis2004roots} and was reformulated by Krithika and Vanchinathan in \cite{krithika2023root}.

\begin{definition}
    
A cluster of $f$ is defined as the subset of $R$ whose elements belong to the field generated by a single root of $f$ over $K$. All the clusters of $f$ have the same cardinality which is defined as the cluster size of $f$ over $K$ denoted as $r_K(f)$. The number of clusters of $f$ over $K$ is denoted as $s_K(f)$.

\end{definition}

For the case in which $K$ is not necessarily perfect, Perlis introduced another notion in \cite{perlisroots} called linear factor quantum number of a polynomial which we now reformulate as follows.

\begin{definition}
 A multicluster of $f$ is defined as the submultiset of $R'$ whose elements belong to the field generated by a single root of $f$ over $K$. All the multiclusters of $f$ have the same cardinality which is defined as the multicluster size of $f$ over $K$ denoted as $l_K(f)$. 

 
\end{definition}

\begin{proposition} \label{perlis}(Reformulation of some results in Perlis \cite{perlis2004roots}, \cite{perlisroots})

\begin{enumerate}
\item $|R'|=deg(f)=n$ and $|R|=Sep\ deg(f)$. 

\item The number of multiclusters of $f$ over $K$ is also $s_K(f)$. The number of distinct fields inside $\bar{K}$ isomorphic to $K(\alpha)/K$ is also $s_K(f)$.

\item  Let $d=|R|$. Then $r_K(f)\cdot s_K(f)=d$ and $l_K(f)\cdot s_K(f)=n$.

\item $r_K(f)=|Aut(K(\alpha)/K)|$ where $\alpha\in \bar{K}$ is any root of $f$.

\end{enumerate}
\end{proposition}

Some observations following from results in VII \S 4 of \cite{lang2012algebra}:

\begin{proposition}
\hfill
\begin{enumerate}

\item $d\mid n$ and $r_K(f)\mid l_K(f)$.

\item If $K$ is perfect then $d=n$ and $r_K(f)=l_K(f)$.
    
\item If $K$ is not perfect and $char(K)=p>0$ then 

\begin{enumerate}
\item $f(x)=h(x^{p^{\mu}})$ where $h$ is a separable polynomial over $K$ of degree $d$ and $n=p^{\mu}\cdot d$ and $l_K(f)=p^{\mu}\cdot  r_K(f)$ for some integer $\mu \geq 0$.

\item Let $\alpha\in \bar{K}$ be any root of $f$ and let $K(\alpha)_S$ be the maximal separable subextension of $K(\alpha)/K$. Then $K(\alpha)_S=K(\alpha^{p^{\mu}})$; $d=[K(\alpha)_S:K]$; $l_K(f)=[K(\alpha):K(\alpha)_S]\cdot r_K(f)$.  

\end{enumerate}

\end{enumerate}
\end{proposition}

For a perfect base field, the notion of cluster size of an extension was defined by the author and Bhagwat in \cite{Bhagwat_2025} which we can similarly define for any separable extension over any base field.

\begin{definition}
Let $L/K$ be a separable extension. By primitive element theorem $L=K(\alpha)$ for some $\alpha\in\bar{K}$. The cluster size of $L/K$ is defined as $r_K(L):= r_K(f)$ where $f$ is minimal polynomial of $\alpha$. This is well defined because of separability (See Section 2.1 in \cite{Bhagwat_2025}). Clearly $r_K(L)=|Aut(L/K)|$ by Proposition \ref{perlis}.\smallskip

We can similarly define number of clusters of $L/K$ as $s_K(L):=s_K(f)$. Therefore $r_K(L) \cdot s_K(L) =  [L : K]$. We also have that $s_K(L)$ is the number of distinct fields inside $\bar{K}$ isomorphic to $L/K$ (See Section 3.2 in \cite{Bhagwat_2025}).

\end{definition}

We now generalize the above notions for any extension.

\begin{definition}
  Let $L/K$ be any extension. We define the cluster size of $L/K$, $r_K(L):= |Aut(L/K)|$. We define the number of clusters of $L/K$ as $s_K(L):=$ the number of distinct fields inside $\bar{K}$ isomorphic to $L/K$.
\end{definition}

We also define the following for any extension.

\begin{definition}
    Consider $L/K$ and let $L_S/K$ be the maximal separable subextension of $L/K$. We define the multicluster size of $L/K$, $l_K(L):= [L:L_S]\cdot r_K(L)$. One can also refer $s_K(L)$ as the number of multiclusters with reference to Proposition \ref{perlis}.
\end{definition}

We have the following interesting properties.

\begin{proposition}
\label{r prop}    Consider $L/K$ and let $L_S/K$ be as above. Let $A=Aut(L/K)$.

    \begin{enumerate}

\item $r_K(L)=[L:L^{A}]$. Also $N=L^{A}$ is the unique intermediate extension $N/K$ such that $L/N$ is Galois of maximum possible degree.\smallskip

\item $l_K(L_S)=r_K(L_S)$ and $r_{L_S}(L)=s_{L_S}(L)=1$ and $l_{L_S}(L)=[L:L_S]$. \smallskip   

\item $r_K(L)\cdot s_K(L)=[L_S:K]$ and $l_K(L)\cdot s_K(L)=[L:K]$. Also $[L^A:K]=[L:L_S]\cdot s_K(L)$. \smallskip

\item   $r_K(L)\mid r_K(L_S)$ and $s_K(L_S)\mid s_K(L)$.\smallskip

\item If $L/K$ is normal then 

\begin{enumerate}
    \item 

 $r_K(L_S)=[L_S :K]$ and $s_K(L_S)=1$.\smallskip

\item $r_K(L)=r_K(L_S)$ and $l_K(L)=[L:K]$ and $s_K(L)=1$.\smallskip

\item $r_K(L^{A})=s_K(L^{A})=1$ and $l_K(L^{A})=[L^{A}:K]$

\end{enumerate}
      
    \end{enumerate}

\end{proposition}

\begin{proof}\hfill
\begin{enumerate}
\item   We have $L/L^{Aut(L/K)}$ to be Galois with Galois group $Aut(L/K)$ (Theorem 2 in VIII \S 1 of \cite{lang2012algebra}). Thus $|Aut(L/K)|=[L:L^{Aut(L/K)}]$. The other assertion follows from Remarks 7.1.2 and 7.1.3 in \cite{Bhagwat_2025}. \smallskip

\item $L/L_S$ is purely inseparable (Proposition 11 in VII \S 7 of \cite{lang2012algebra}). Thus $|Aut(L/L_S)|=1$.\smallskip

\item Let $E$ be the set of embeddings of $L$ in $\bar{K}$ fixing $K$. Now $|E|=[L_S:K]$ (Corollary 3 in VII \S 7 of \cite{lang2012algebra}). Consider the right action of $Aut(L/K)$ on $E$. This action is free and any orbit is precisely the subset of $E$ consisting of embeddings having the same image inside $\bar{K}$. Thus the number of distinct fields inside $\bar{K}$ isomorphic to $L/K$ is precisely $|E|/|Aut(L/K)|$.

\smallskip

\item Let $\iota: Aut(L/K)\rightarrow Aut(L_S/K)$ be defined by $\iota(\sigma):=\sigma|_{L_S}$. Then $\iota$ is a well defined injective homomorphism. Also $r_K(L)\cdot s_K(L)=[L_S:K]=r_K(L_S)\cdot s_K(L_S)$.\smallskip

\item Now $L/K$ is normal. Thus $L_S/K$ is Galois (Corollary 2 in VII \S 7 of \cite{lang2012algebra}). Thus $|Aut(L_S/K)|=[L_S:K]$. Since $L/K$ is normal we have $\iota$ in part (4) to be an isomorphism. Also normality of $L/K$ implies that $L^{A}/K$ is purely inseparable (Proposition 12 in VII \S 7 of \cite{lang2012algebra}).

\end{enumerate}    
\end{proof}

\begin{corollary}
\label{cor}   
       The extension $L/K$ is normal $\iff s_K(L)=1\iff r_K(L)=[L_S:K]\iff l_K(L)=[L:K]\iff [L:L_S]=[L^A:K]\iff L^A/K$ is purely inseparable.

\end{corollary}

\begin{proof}
    We will prove $L^A/K$ is purely inseparable $\implies L/K$ is normal. Firstly we claim that for any extension $L/K$, we have $L=L_S L^A$. This is because $L/L_S$ is purely inseparable and $L/L^A$ is Galois, thus separable, which implies that $L/(L_S L^A)$ is both purely inseparable and separable. Now since $L^A/K$ is purely inseparable, by Corollary 3 in VII \S 7 of \cite{lang2012algebra}, $[L^A:K]=[L_SL^A:L_S]=[L:L_S]$.
\end{proof}

For a perfect base field, the notion of ascending index of an extension (as a dual notion to cluster size of an extension) was defined by the author and Bhagwat in \cite{Bhagwat_2025} which we can similarly define for any extension over any base field.

\begin{definition}\label{asc ind def}
    Let $L/K$ be any extension over any base field. We define the ascending index of $L/K$ as $t_K(L):= [F:K]$ where $F/K$ is the largest Galois subextension (hence unique) of $L/K$. Clearly $F\subset L_S$. Let $u_K(L):=[L_S:F]$.
\end{definition}

\begin{proposition}\label{asc ind}
Consider $L/K$.

    \begin{enumerate}
 \item $t_K(L)\cdot u_K(L)=[L_S:K]$ and $t_K(L)=t_K(L_S)$ and $u_K(L)=u_K(L_S)$. We also have $t_{L_S}(L)=u_{L_S}(L)=1$.\smallskip

 \item If $L/K$ is normal then  

 \begin{enumerate}
     \item $t_K(L)=[L_S:K]$ and $u_K(L)=1$.\smallskip

     \item $t_K(L^{A})=u_K(L^{A})=1$.
 \end{enumerate}
 
    \end{enumerate}
\end{proposition}

We also define the following general notions.

\begin{definition}\label{dni ani}
      Let $L/K$ be any extension. We define the descending normal index of $L/K$ as $d_K(L):= [L:N']$ where $N'/K$ is a subextension of $L/K$ such that $L/N'$ is normal  with maximum possible degree.\smallskip
      
      We also define the ascending normal index of $L/K$ as $a_K(L):= [F':K]$ where $F'/K$ is a subextension of $L/K$ such that $F'/K$ is normal  with maximum possible degree.
\end{definition}

\begin{proposition}
\label{ad prop}
 Consider $L/K$ and let $L_I/K$ be the maximal purely inseparable subextension of $L/K$ and $i_K(L)=[L_I:K]$. Let $F'/K$ be the maximal normal subextension of $L/K$ and $F/K$ be the largest Galois subextension of $L/K$. Then
    \begin{enumerate}

\item $F'=FL_I$ and $a_K(L)=t_K(L)\cdot i_K(L)$.

\item $a_K(L)\cdot u_K(L)=[L_SL_I:K]$.

\item $L/K$ is normal $\iff a_K(L)=[L:K]$. 

    \end{enumerate}
\end{proposition}

\begin{proof}
    
\hfill

\begin{enumerate}

\item   Clearly $F\subset F'$. Also since $L_I/K$ is the maximal purely inseparable subextension of $L/K$, it is a normal subextension. Thus $L_I\subset F'$. Hence $FL_I\subset F'$. By Corollary 3 in VII \S 7 of \cite{lang2012algebra}, we have that $[FL_I:K]=[F:K]\cdot [L_I:K]=t_K(L)\cdot i_K(L)$. Now by Proposition 12 in VII \S 7 of \cite{lang2012algebra}, $F'=F'_S F'^{Aut(F'/K)}$ where $F'_S$ is the maximal separable subextension of $F'/K$ and $F'^{Aut(F'/K)}$ is purely inseparable over $K$. Clearly $F'_S\subset L_S$ and $F'^{Aut(F'/K)}\subset L_I$. Now $F'_S$ is Galois by Corollary 2 in VII \S 7 of \cite{lang2012algebra}. Thus $F'_S\subset F$. Hence $F'\subset FL_I$ and therefore $F'=FL_I$. Thus also $F'_S=F$ and $F'^{Aut(F'/K)}=L_I$.

\smallskip

\item By Corollary 3 in VII \S 7 of \cite{lang2012algebra}, we have $[L_SL_I:L_S]=[L_I:K]$.\smallskip

\item This is trivial.

\end{enumerate}
\end{proof}

\begin{corollary}\label{LI cor}
      $L/K$ is normal $\iff$ $t_K(L)=[L:L_I]$ $\iff$ ($u_K(L)=1$ and $L=L_S L_I$) $\iff$ $L^A=L_I$ (where $A=Aut(L/K)$) $\iff$ $L/L_I$ is Galois.
\end{corollary}

\begin{proof}
     We will show that $L^A=L_I$ $\iff$ $L^A$ is purely inseparable. Then by Corollary \ref{cor} we will have $L^A=L_I\iff L/K$ is normal. Suppose $L^A/K$ is  purely inseparable. Thus $L^A\subset L_I$. Since $L/L^A$ is Galois, so $L_I/L^A$ is both purely inseparable and separable. Therefore $L^A=L_I$. \smallskip
     
     Now suppose $L/L_I$ is Galois. As $Aut(L/L_I)\subset Aut(L/K)$, $L_I=L^{Aut(L/L_I)}\supset L^A$. Therefore $L^A=L_I$.
\end{proof}

\begin{remark}\label{LI in LA}
    Note that $i_{L_I}(L)=1$. Also note that $L_I\subset L^A$ is always true. If $x\in L_I$ then there is a $\mu\geq 0$ such that $x^{p^{\mu}}\in K$. Thus for any $\sigma\in Aut(L/K)$, $\sigma(x^{p^{\mu}})=x^{p^{\mu}}$. So $\sigma(x)^{p^{\mu}}=x^{p^{\mu}}$. Thus $(\sigma(x)-x)^{p^{\mu}}=0$. Thus $\sigma(x)=x$ for any $\sigma\in A$ and so $x\in L^A$. 
\end{remark}

\begin{lemma}
    \label{normal intersection}

    Consider $L/K$ and subextensions $K_1/K$ and $K_2/K$.
    \begin{enumerate}
\item If $L/K_1$ and $L/K_2$ are Galois, then $L/(K_1\cap K_2)$ is Galois.
        \smallskip

\item If $L/K_1$ and $L/K_2$ are normal, then $L/(K_1\cap K_2)$ is normal.
    \end{enumerate}
\end{lemma}

\begin{proof}
    \hfill

    \begin{enumerate}
\item Let $A=Aut(L/K)$ and $A_1=Aut(L/K_1)$ and $A_2=Aut(L/K_2)$. Now $A_1,A_2\subset A$. Also $L/L^A$ is Galois. Since $L/K_1$ and $L/K_2$ are Galois, $K_1=L^{A_1}\supset L^A$ and $K_2=L^{A_2}\supset L^A$. Thus $K_1\cap K_2\supset L^A$. Thus $L/(K_1\cap K_2)$ is Galois.

\smallskip

\item Let $L'_I/K_1$ and $L''_I/K_2$ be the respective maximal purely inseparable subextensions of $L/K_1$ and $L/K_2$. Since $L/K_1$ and $L/K_2$ are normal, by Corollary \ref{LI cor} we have $L^{A_1}=L'_I$ and $L^{A_2}=L''_I$ where $A_1$ and $A_2$ are as in previous part. Thus $L/L'_I$ and $L/L''_I$ are Galois. By part (1), $L/(L'_I\cap L''_I)$ is Galois. Let $x\in L'_I\cap L''_I$. We can choose a large integer $\mu$ such that $x^{p^{\mu}}\in K_1\cap K_2$. Thus $(L'_I\cap L''_I)/(K_1\cap K_2)$ is purely inseparable. Let $L'''_I$ be maximal purely inseparable subextension of $L/(K_1\cap K_2)$. Then $L'_I\cap L''_I\subset L'''_I$ and $L'''_I/(L'_I\cap L''_I)$ is both purely inseparable and separable. Thus $L'''_I=L'_I\cap L''_I$. As $L/L'''_I$ is Galois, by Corollary \ref{LI cor} we have that $L/(K_1\cap K_2)$ is normal.

    \end{enumerate}
\end{proof}

\begin{lemma}
    Consider separable extension $S/K$ and purely inseparable extension $I/K$. Then $S/K$ is Galois $\iff$ $SI/I$ is Galois.
\end{lemma}

\begin{proof}
     Since $S/K$ is separable, by primitive element theorem, $S=K(\alpha)$ for some $\alpha\in S$. Let $f$ be minimal polynomial of $\alpha$ over $K$.
     Now $f$ is separable. Let $deg(f)=k$ and $\alpha_1,\dots,\alpha_k$ be all the distinct roots of $f$ in $\bar{K}$. Now since $[SI:I]=[S:K]$. So $[I(\alpha):I]=[K(\alpha):K]$. Thus $f$ is also the minimal polynomial of $\alpha$ over $I$. Suppose $S/K$ is Galois. Then all $\alpha_i\in S$. Thus all $\alpha_i\in SI$. So $SI/I$ is Galois. Conversely suppose $SI/I$ is Galois. Thus all $\alpha_i\in SI$. Since $SI/S$ is purely inseparable, the maximal separable subextension of $SI/K$ is precisely $S/K$. Also since all $\alpha_i$ are separable over $K$ and contained in $SI$, we have all $\alpha_i\in S$. So $S/K$ is Galois.
\end{proof}

\begin{proposition}\label{L'S prop}
    Consider any $L/K$ with $A=Aut(L/K)$. As noted in Remark \ref{LI in LA}, $L_I\subset L^A$.  Let $(L^A)_S$ and $L'_S$ be respective maximal separable subextensions of $L^A/K$ and $L^A/L_I$. Then 

    \begin{enumerate}
    
\item $(L^A)_S=L_S\cap L^A$ and $L_S/(L^A)_S$ is Galois, $Gal(L_S/(L^A)_S)\cong A$ and $[L_S:(L^A)_S]=r_K(L)$ and $[(L^A)_S:K]=s_K(L)$ and $[L:L_S]=[L^A:(L^A)_S]$. Also $r_K(L)\mid r_K(L_S)$.
\smallskip
    
\item $L'_S=L_S L_I\cap L^A$ and $(L_S L_I)/ L'_S$ is Galois, $Gal((L_S L_I)/L'_S)\cong A$ and $[L_S L_I : L'_S]=r_K(L)$ and $[L'_S:L_I]=s_K(L)$ and $[L:L_S L_I]=[L^A:L'_S]$ and $L'_S\cap L_S=(L^A)_S$ and $(L^A)_S L_I=L'_S$.\smallskip

\item $L=L_SL_I$ $\iff$ $L^A=L'_S$ $\iff$ $L^A/L_I$ is separable $\iff$ $[L:L_S]=i_K(L)$.\smallskip

\item  $L=L_SL_I$ $\implies$ $(L^A)_S=(L_S)^{A'}$ where $A'=Aut(L_S/K)$ and thus $r_K(L)=r_K(L_S)$.

    \end{enumerate}

\end{proposition}

\begin{proof}

\hfill

\begin{enumerate}

\item Clearly $(L^A)_S=L_S\cap L^A$. Now $L_S/(L^A)_S$ is separable and $L^A/(L^A)_S$ is purely inseparable and $L=L_S L^A$ and $L/L^A$ is Galois with $Gal(L/L^A)\cong A$. So the assertions follows. For showing $r_K(L)\mid r_K(L_S)$, we observe that as $L_S/(L^A)_S$ is Galois, $(L_S)^{A'}\subset (L^A)_S$ where $A'=Aut(L_S/K)$.

\smallskip

    \item   As $L_S L_I/L_I$ is separable, $(L_S L_I\cap L^A)/L_I$ is separable. Thus $L_S L_I\cap L^A\subset L'_S$. As $L'_S / L_I$ is separable, $(L_S L'_S)/(L_S L_I)$ is separable. As $L/L_S$ is purely inseparable, $L/L_S L_I$ is purely inseparable and thus $(L_S L'_S)/(L_S L_I)$ is also purely inseparable. So $L_S L'_S=L_S L_I$. So $L'_S\subset L_SL_I$ and thus $L'_S=L_S L_I\cap L^A$.\smallskip

Alternatively this can be proven by applying part (1) to the extension $L/L_I$ by observing that $L_SL_I/L_I$ is maximal separable subextension of $L/L_I$ and that $L^{Aut(L/L_I)}=L^A$ as $Aut(L/L_I)=A$ and $[L_SL_I:L_I]=[L_S:K]$. The other assertions follow. \smallskip

\item  Follows from part (2).

\smallskip

\item Suppose $L=L_S L_I$. Now $(L_S)^{A'}/K$ is separable, so $(L_S)^{A'}L_I/(L_S)^{A'}$ is purely inseparable with degree $[I:K]$. As $L_S/(L_S)^{A'}$ is Galois. Therefore $L/(L_S)^{A'}L_I$ is Galois with $Gal(L/(L_S)^{A'}L_I)\cong Gal(L_S/(L_S)^{A'})$. Therefore $L^A\subset (L_S)^{A'}L_I$. Hence $r_K(L_S)=[L_S:(L_S)^{A'}]=[L:(L_S)^{A'}L_I]\mid [L:L^A]= r_K(L)$. By part (1), $r_K(L)\mid r_K(L_S)$. Thus $r_K(L)=r_K(L_S)$ and $L^A= (L_S)^{A'}L_I$. Also $(L_S)^{A'}\subset (L^A)_S$ and by part (1), $[L_S:(L^A)_S]=r_K(L)$. Thus $(L^A)_S=(L_S)^{A'}$. 

    \end{enumerate}
\end{proof}


\begin{proposition} \label{unique desc normal}
  Consider $L/K$. Then
  
\begin{enumerate}

\item There is a unique subextension $N'/K$ such that $L/N'$ is normal with $[L:N']=d_K(L)$. 

\smallskip

\item If $N''/K$ is a subextension such that $L/N''$ is normal, then $N''\supset N'$.

\smallskip

\item  $N'=(L^A)_S$ (where $A=Aut(L/K)$) and $d_K(L)=r_K(L)\cdot [L:L_S]=l_K(L)$ and we have $d_K(L)\cdot s_K(L)=[L:K]$. Thus the multicluster size is the descending normal index.

\smallskip

\item $L/K$ is normal $\iff d_K(L)=[L:K]$.
    \end{enumerate}
\end{proposition}

\begin{proof}

\hfill

\begin{enumerate}
    \item  If $N_1/K$ is another subextension such that $L/N_1$ is normal with $[L:N_1]=d_K(L)$, then by Lemma \ref{normal intersection} (2), we have $L/(N_1\cap N')$ is also normal and $[L:(N_1\cap N')]\geq d_K(L)$. By Definition \ref{dni ani}, $[L:(N_1\cap N')]= d_K(L)$. Thus $N_1=N_1\cap N'=N'$.\smallskip

    \item Now if $N''/K$ is a subextension such that $L/N''$ is normal, then $L/(N'\cap N'')$ is also normal. Thus $N'\cap N''=N'$. Therefore $N''\supset N'$.\smallskip

\item By Proposition \ref{L'S prop} (1), $L_S/(L^A)_S$ is Galois. Also $L^A/(L^A)_S$ is purely inseparable and $L=L_SL^A$. Thus $L/(L^A)_S$ is normal. By part (2), $(L^A)_S\supset N'$. Also $L^A\supset N'$. Thus the unique intermediate extension $N/N'$ of $L/N'$ such that $L/N$ is Galois of maximum possible degree is $N=L^{Aut(L/N')}=L^A$. Since $L/N'$ is normal, $L^A/N'$ is purely inseparable by Corollary \ref{cor}. Thus $(L^A)_S/N'$ is both purely inseparable and separable. Thus $N'=(L^A)_S$.\smallskip

Thus $d_K(L)=[L:N']=[L:(L^A)_S]=[L_S:(L^A)_S]\cdot [L:L_S]$ and $[L_S:(L^A)_S]=r_K(L)$ by Proposition \ref{L'S prop} (1). The other assertions follow.

\smallskip

\item This is trivial.
\end{enumerate}
    
\end{proof}

\begin{lemma}\label{compositum}
    Consider $L/K$ and $L'/K$. Then 

    \begin{enumerate}
        \item $(L\cap L')_S=L_S \cap L'_S$ and $(LL')_S=L_S L'_S$.

        \item $(L\cap L')_I=L_I \cap L'_I$. Suppose $L/K$ and $L'/K$ are normal, then $(LL')_I=L_I L'_I$.
    \end{enumerate}
\end{lemma}

\begin{proof} The result is trivial when $K$ is perfect. Suppose it is not perfect and $char(K)=p>0$.

\begin{enumerate}
    \item We prove the second assertion. Clearly $L_S L'_S \subset (LL')_S$. Now let $\Sigma\ l_i l'_i\in LL'$ where each $l_i\in L, l'_i\in L'$. We can choose a large integer $\mu$ such that each $l_i^{p^\mu}\in L_S$ and ${l'}_i^{p^{\mu}}\in L'_S$. Thus $(\Sigma\ l_il'_i)^{p^{\mu}}=\Sigma\ l_i^{p^{\mu}} {l'}_i^{p^{\mu}} \in L_S L'_S$. Hence $LL'/ L_SJ_S$ is purely inseparable. Thus $(LL')_S /L_SJ_S$ is purely inseparable. But $(LL')_S /L_SJ_S$ is also separable, so $(LL')_S=L_S L'_S$. 
\smallskip

\item
We prove the second assertion. As $L/K$ and $L'/K$ are normal, $LL'/K$ is normal. By Corollary \ref{LI cor}, $L^A=L_I$ and $L=L_S L_I$. Similarly we have $L'=L'_S L'_I$ and $LL'=(LL')_S (LL')_I$. Thus $(L_S L_I)(L'_S L'_I)=(LL')_S (LL')_I$.\smallskip

Now clearly $L_I L'_I\subset (LL')_I$. From Part (1), $(LL')_S=L_S L'_S$. Thus $(LL')_S(L_I L'_I)=(LL')_S (LL')_I$. By Corollary 3 in VII \S 7 in \cite{lang2012algebra}, it follows that $[(LL')_S:K][(L_I L'_I):K]=[(LL')_S:K][(LL')_I:K]$. Therefore $[(L_I L'_I):K]=[(LL')_I:K]$. Hence $L_I L'_I=(LL')_I$.\smallskip

Alternatively, one can observe that as $L/L_I$ and $L'/L'_I$ are separable, any element of $L$ or $L'$ is also separable over $L_I L'_I$. Thus $LL'/L_I L'_I$ is separable. Hence $(LL')_I/ L_IL'_I$ is both separable and purely inseparable.\end{enumerate}\end{proof}

Now we generalize Proposition 2.2 in \cite{jaiswal2025rootcapacityintersectionindicium}.

\begin{proposition}\label{f'f}
Consider $L/K$. Let $F'/K$ be the largest normal subextension of $L/K$ and $F/K$ be the largest Galois subextension of $L/K$. Then  

\begin{enumerate}
    \item 
   $\tilde{L}=\Pi_{i=1}^s L_i$ and $F'=\cap_{i=1}^{s} L_i$ where $\tilde{L}$ is normal closure of $L/K$ and $L_i$'s are the $s=s_K(L)$ many distinct fields isomorphic to $L$ over $K$.

\smallskip

    \item 
   $\tilde{(L_S)}=\Pi_{j=1}^{s'} (L_S)_j$ and $F=\cap_{j=1}^{s'} (L_S)_j$ where $\tilde{(L_S)}$ is Galois closure of $L_S/K$ and $(L_S)_j$'s are the $s'=s_K(L_S)$ many distinct fields isomorphic to $L_S$ over $K$.

\smallskip

    \item $\tilde{(L_S)}=\Pi_{i=1}^{s} (L_i)_S$ and $F=\cap_{i=1}^{s} (L_i)_S$ where $L_i$'s are as above.

    \smallskip

    \item $(\tilde{L})_S=\tilde{(L_S)}$ where $(\tilde{L})_S$ is the maximal separable subextension of $\tilde{L}/K$.

    \smallskip

    \item $F'_S=F$.
    \smallskip

    \item For any $i$, $(L_i)_I=L_I$.\smallskip

    \item If $L=L_SL_I$ then $(\tilde{L})_I=L_I$.
    
\end{enumerate}
\end{proposition}

\begin{proof}\hfill
    \begin{enumerate}
    
        \item 
        Clearly $L_i\subset \tilde{L}$ for each $i$. Thus $\Pi_{i=1}^s L_i\subset \tilde{L}$. It is enough to show that $(\Pi_{i=1}^s L_i)/K$ is normal. Consider an embedding $\sigma: \Pi_{i=1}^s L_i\rightarrow \tilde{L}$ fixing $K$. Now $\sigma$ permutes $L_i$'s. Thus $\sigma (\Pi_{i=1}^s L_i)=\Pi_{i=1}^s L_i$.

        \smallskip
        
       Now consider the $K$-isomorphism $\sigma_i:L\rightarrow L_i$ for each $i$. Since $F'/K$ is normal, we have $\sigma_i(F')=F'$ for each $i$. Thus $F'\subset L_i$ for each $i$. Thus $F'\subset \cap_{i=1}^s L_i$. It is enough to show that $(\cap_{i=1}^s L_i)/K$ is normal. Consider an embedding $\sigma: \cap_{i=1}^s L_i\rightarrow \tilde{L}$ fixing $K$. We can extend this to a $K$-isomorphism $\tilde{\sigma}: \tilde{L}\rightarrow \tilde{L}$. Now $\tilde{\sigma}$ permutes $L_i$'s. Let $l\in \cap_{i=1}^s L_i$. So $\sigma(l)\in \cap_{i=1}^s \tilde{\sigma}(L_i)=\cap_{i=1}^s L_i$.

     \smallskip  
     
     \item Since $F/K$ is the largest normal subextension of $L_S/K$, the result follows from Part (1).\smallskip

        \item   It is enough to show that the sets $A=\{(L_i)_S\}_{i=1}^s$ and $B=\{(L_S)_j\}_{j=1}^{s'}$ are equal. Consider $(L_i)_S\in A$ and consider the $K$-isomorphism $\sigma : L\rightarrow L_i$. One can show that $\sigma(L_S)= (\sigma(L))_S$. Thus $\sigma(L_S)=(L_i)_S$. Thus $(L_i)_S\in B$.\smallskip

        Conversely, consider $(L_S)_j\in B$ and consider the $K$-isomorphism $\sigma: L_S\rightarrow (L_S)_j$. Since $L/L_S$ is purely inseparable and thus normal, so by Theorem 3.20 in \cite{morandi2012field} we can find a normal extension $L'/(L_S)_j$ with an isomorphism $\tilde{\sigma}: L\rightarrow L'$ such that $\tilde{\sigma}|_{L_S}=\sigma$. Thus $\tilde{\sigma}$ is a $K$-isomorphism. Thus $L'/K$ is isomorphic to $L/K$. Also $(L')_S=(\tilde{\sigma}(L))_S=\tilde{\sigma}(L_S)=\sigma(L_S)=(L_S)_j$. Thus $(L_S)_j\in A$.

\smallskip

        \item By Part (1) and Lemma \ref{compositum} (1), $(\tilde{L})_S=\Pi_{i=1}^s (L_i)_S$. Thus by Part (3) we are done.
\smallskip

 \item Follows from Lemma \ref{compositum} (1). Also follows from proof of Proposition \ref{ad prop} (1).\smallskip
 
 \item Consider the $K$-isomorphism $\sigma : L\rightarrow L_i$. Since $L_I/K$ is normal, we have $\sigma(L_I)=L_I$. Thus $L_I\subset (L_i)_I$. Similarly we have $(L_i)_I\subset L_I$.\smallskip

 \item As $L=L_SL_I$, we have $s_K(L)=s_K(L_S)$ by Proposition \ref{L'S prop} (4) which we denote by $s$. Let $\{(L_S)_i\}_{i=1}^s$ be all the distinct fields isomorphic to $L_S/K$. Then $\{(L_S)_i L_I\}_{i=1}^s$ are precisely all the distinct fields isomorphic to $L/K$. Thus $\tilde{L}=\Pi_{i=1}^s ((L_S)_iL_I)= (\Pi_{i=1}^s (L_S)_i)L_I=\tilde{(L_S)}L_I$. Thus $\tilde{L}/L_I$ is Galois. Clearly $L_I\subset (\tilde{L})_I$. So $(\tilde{L})_I/L_I$ is both purely inseparable and separable, so we are done.
 
 \end{enumerate}\end{proof}


 \begin{proposition}\label{s/s'}
     Consider $L/K$. The number of distinct extensions $L'/K$ that are isomorphic to $L/K$ such that $L'_S=L_S$ is exactly $s_K(L)/s_K(L_S)$.
 \end{proposition}

 \begin{proof}
Let $L_i$'s be the $s=s_K(L)$ many distinct fields isomorphic to $L$ over $K$ and $S=\{L_i\}_{i=1}^s$ and $(L_S)_j$'s be the $s'=s_K(L_S)$ many distinct fields isomorphic to $L_S$ over $K$. We partition $S$ into $s'$ many disjoint subsets $S_j=\{L_i\in S\mid (L_i)_S=(L_S)_j\}$. By proof of Proposition \ref{f'f} (3), each $S_j\neq \emptyset$. We  will show that each $S_j$ has same cardinality and hence $|S_j|=s_K(L)/s_K(L_S)$ for each $j$. By relabeling, let $L=L_1,\dots,L_m$ be all the distinct fields $L'/K$ that are isomorphic to $L/K$ such that $L'_S=L_S$. Let $L''/K$ be a field isomorphic to $L/K$ such that $L''_S\neq L_S$. Consider the $K$-isomorphism $\sigma : L\rightarrow L''$. Then $\sigma(L_S)=L''_S$. We can extend $\sigma$ to a $K$-isomorphism $\tilde{\sigma}: \tilde{L}\rightarrow \tilde{L}$ with $\tilde{\sigma}|_{L}=\sigma$. Now for $1\leq i\leq m$, $\tilde{\sigma}(L_i)$ are distinct fields with $(\tilde{\sigma}(L_i))_S=\tilde{\sigma}((L_i)_S)=\tilde{\sigma}(L_S)=\sigma(L_S)=L''_S$. If $m'$ is the number of all distinct fields isomorphic to $L/K$ with maximal separable subextension $L''_S$, then $m\leq m'$. Similarly we can show $m'\leq m$, so we are done.


\smallskip

     Alternatively, consider the notations as in proof of Proposition \ref{r prop} (3). Now embeddings of $L$ in $\bar{K}$ fixing $K$ are in one to one correspondence with embeddings of $L_S$ in $\bar{K}$ fixing $K$. Thus the embeddings $\sigma : L\rightarrow \bar{K}$ such that $\sigma_K=id_K$ and $\sigma(L_S)=L_S$ are in one to one correspondence with elements in $Aut(L_S/K)$. Thus the number of distinct fields $L'/K$ that are isomorphic to $L/K$ such that $L'_S=L_S$ is $|Aut(L_S/K)|/|Aut(L/K)|$. Now by Proposition \ref{r prop} (4), $r_K(L_S)/r_K(L)=s_K(L)/s_K(L_S)$.
 \end{proof}

 \begin{proposition}
  Consider finite $L/K$ and let $\tilde{L}/K$ be its normal closure. Then we have $[\tilde{L}:K]\mid ([L_S:K]!\cdot [L:L_S]^{s_K(L)})$. Thus also $[\tilde{L}:K]\mid [L:K]!$.
 \end{proposition}

 \begin{proof}
As $L_S/K$ is separable, $[\tilde{(L_S)}:K]\mid [L_S:K]!$. Now $(\tilde{L})_S=\tilde{(L_S)}$ by Proposition \ref{f'f} (4). So $[(\tilde{L})_S:K]\mid [L_S:K]!$. Now we will show that $[\tilde{L}: (\tilde{L})_S]\mid [L:L_S]^{s_K(L)}$ which will prove the result. Now $\tilde{L}=\Pi_{i=1}^{s_K(L)} L_i$ where $L_i$'s are all the distinct fields isomorphic to $L$ over $K$. Also $\tilde{(L_S)}=(\tilde{L})_S=\Pi_{i=1}^{s_K(L)} (L_i)_S$. Let $s=s_K(L)$ and observe that $(L_1\cdots L_{s-1})/(L_1\cdots L_{s-1})_S$ is purely inseparable and $(\tilde{L})_S/(L_1\cdots L_{s-1})_S$ is separable. Thus $[(L_1\cdots L_{s-1})(\tilde{L})_S:(\tilde{L})_S]=[(L_1\cdots L_{s-1}):(L_1\cdots L_{s-1})_S]$. Similarly $[L_s (\tilde{L})_S:(\tilde{L})_S]=[L_s: (L_s)_S]=[L:L_S]$. Therefore we have $[\tilde{L}: (\tilde{L})_S]\mid [(L_1\cdots L_{s-1})(\tilde{L})_S:(\tilde{L})_S]\cdot [L_s (\tilde{L})_S:(\tilde{L})_S]=[(L_1\cdots L_{s-1}):(L_1\cdots L_{s-1})_S]\cdot [L:L_S]$. By induction, $[(L_1\cdots L_{s-1}):(L_1\cdots L_{s-1})_S]\mid [L:L_S]^{s-1}$. Thus $[\tilde{L}: (\tilde{L})_S]\mid [L:L_S]^s$.\smallskip

Clearly, $([L_S:K]!\cdot [L:L_S]^{s_K(L)})\mid ([L_S:K]!\cdot [L:L_S]^{[L_S:K]})\mid ([L_S:K]!\cdot ([L:L_S]!)^{[L_S:K]}) \mid ([L_S:K]\cdot [L:L_S])!=[L:K]!$.
 \end{proof}

 \begin{proposition}
    Let $L/K$ be a finite extension. Suppose $L$ is generated by $k$ elements over $K$. Then any subextension of $L/K$ can be generated by $\leq k$ elements over $K$.
 \end{proposition}

 \begin{proof}
   Let $L'/K$ be a subextension of $L/K$. As $L/K$ is generated by $k$ elements, we have that $L/L_S$ is also generated by the same $k$ elements. Now $L/L_S$ is purely inseparable and $L'L_S/L_S$ is a subextension of $L/L_S$. By Corollary of Theorem 6.3 in \cite{gerstenhaber1981geometry}, $L'L_S/L_S$ is generated by $l\leq k$ elements. Now $L_S/L'_S$ is separable and $L'/L'_S$ is purely inseparable. Thus by Lemma 6.8 (and the remark that follows) in \cite{gerstenhaber1981geometry}, $L'/L'_S$ is  generated by $m\leq l$ elements say $\alpha_1,\dots,\alpha_m$. By primitive element theorem $L'_S=K(\beta)$ for some $\beta\in \bar{K}$. Thus $L'=K(\beta, \alpha_1, \dots, \alpha_m)$. By stronger version of primitive element theorem, Theorem 1 in \cite{brown2012mathematics} (or by Theorem C.1 in \cite{conradseparability}), $K(\beta, \alpha_1)/K$ is a simple extension, say $K(\gamma)/K$. Thus $L'=K(\gamma, \alpha_2, \dots, \alpha_m)$. Therefore $L'/K$ is generated by $m$ elements where $m\leq k$.
 \end{proof}

Generalizations of certain statements in \cite{Bhagwat_2025} and \cite{jaiswal2025rootcapacityintersectionindicium}.

\begin{proposition} Consider $M/L/K$. Then 
    \begin{enumerate}
\item $r_L(M)\mid r_K(M)$ and $l_L(M)\mid l_K(M)$.\smallskip

\item $t_K(L)\mid t_K(M)$ and $i_K(L)\mid i_K(M)$ and $a_K(L)\mid a_K(M)$.
\smallskip

\item $t_K(M)\mid (t_L(M)\cdot [L:K])$ and $i_K(M)\mid (i_L(M)\cdot [L:K])$ and $a_K(M)\mid (a_L(M)\cdot [L:K])$.\smallskip

\item Suppose $M/K$ is normal and $M/L$ is Galois. Then $$r_K(L)=[N_{Aut(M/K)}(Gal(M/L)): Gal(M/L)].$$

    \end{enumerate}
\end{proposition}

\begin{proof}\hfill

\begin{enumerate}
    \item We have $Aut(M/L)\subset Aut(M/K)$. Thus $|Aut(M/L)|\mid |Aut(M/K)|$. Let $M_S/K$ and $M'_S/L$ be the respective maximal separable subextensions of $M/K$ and $M/L$. Now $M_SL/L$ is also separable. Thus $M_S\subset M_SL\subset M'_S$. So $[M:M'_S]\mid [M:M_S]$. Alternatively, $l_K(M)=d_K(M)$ and we have $d_L(M)\mid d_K(M)$ by Proposition \ref{unique desc normal} (2). \smallskip

    \item Let $F/K$ and $F_0/K$ be the respective largest Galois subextensions of $L/K$ and $M/K$. Clearly $F\subset F_0$. Thus $[F:K]\mid [F_0:K]$. Also $L_I\subset M_I$ where $M_I$ is the maximal purely inseparable subextension of $M/K$. Let $F'/K$ and $F'_0/K$ be the respective largest normal subextensions of $L/K$ and $M/K$. Clearly $F'\subset F'_0$. \smallskip

\item We show $a_K(M)\mid (a_L(M)\cdot [L:K])$. Let $F'/L$ and $F'_0/K$ be the respective maximal normal subextensions of $M/L$ and $M/K$. Now $F'_0L/L$ is also normal. Thus $F'_0\subset F'_0L\subset F'$. Hence $[F'_0:K]\mid ([F':L]\cdot [L:K])$. The other assertions can be similarly proved.

    \smallskip

\item  By Proposition 6.1.3 in \cite{Bhagwat_2025}, we have the injective map \[\Phi: N_{{\rm Aut}(M/K)}({\rm Aut}(M/L))/ {\rm Aut}(M/L) \xhookrightarrow{}{\rm Aut}(M^{{\rm Aut}(M/L)}/K)\] mapping $\sigma$ to $\sigma|_{M^{Aut(M/L)}}$. As $M/K$ is normal, $\Phi$ is an isomorphism. As $M/L$ is Galois, $M^{Aut(M/L)}=L$. 
\end{enumerate}
    
\end{proof}

    


An interesting example.

\begin{example}\label{interesting eg}

We will construct an $L/K$ such that $L_S/K$ is a nontrivial Galois extension but $L/K$ is not normal. Let $p>2$ be a prime and $K=\mathbb{F}_p(x,y)$ where $x$ and $y$ are indeterminates. Consider thr polynomial $f(t)=t^2+xt+y\in K[t]$ and let $\alpha,\beta\in \bar{K}$ be its roots. Thus $\alpha+\beta=-x$ and $\alpha\beta=y$. We claim that $\alpha\neq \beta$. Suppose $\alpha=\beta$. Then $x^2=4y$ which contradicts algebraic independence of $x$ and $y$ over $K$. We claim that $\alpha,\beta\not \in K$. Suppose $\alpha\in K$ or $\beta\in K$. Then $\sqrt{x^2-4y}\in K$. Let $\sqrt{x^2-4y}=g(x,y)/h(x,y)$ where $g(x,y), h(x,y)\in \mathbb{F}_p[x,y]=A$. Now $A$ is a UFD. Also we can assume that $g$ and $h$ have no common factors. Now $A/(x^2-4y)\cong \mathbb{F}_p[x]$. Thus $x^2-4y$ is prime in $A$. As $g(x,y)^2=(x^2-4y)h(x,y)^2$, we have $(x^2-4y) \mid g(x,y)$ and so $(x^2-4y) \mid h(x,y)$ which is a contradiction. Thus $K(\alpha)/K$ is a quadratic Galois extension.\smallskip

    Consider $\theta=\alpha^{1/p},\phi=\beta^{1/p} \in \bar{K}$. We claim that $\theta, \phi\not \in K(\alpha)$. Suppose $\theta\in K(\alpha)=K(\delta)$ where $\delta=\sqrt{x^2-4y}$. Thus $\theta=\alpha^{1/p}=a+b\delta$ where $a,b\in K$. Hence $\alpha=a^p+b^p \delta^p$. Thus $\alpha=\frac{-x}{2}+\frac{\delta}{2}=a^p+b^p(\delta^2)^{(p-1)/2} \delta$. Therefore $\frac{-x}{2}=a^p$ and $\frac{1}{2}=b^p(\delta^2)^{(p-1)/2}$. Let $a=g_1(x,y)/h_1(x,y)$ where $g_1,h_1\in A$ have no common factors. Now $x$ is prime in $A$. As $-x h_1^p=2 g_1^p$, we have $x\mid g_1$ and so $x\mid h_1$ which is a contradiction. We get a similar contradiction by assuming $\phi\in K(\alpha)=K(\beta)$. Therefore $K(\theta)/K(\alpha)$ and $K(\phi)/K(\alpha)$ are degree $p$ purely inseparable extensions.\smallskip

    Let $L=K(\theta)$. Thus $[L:K]=[L:K(\alpha)]\cdot [K(\alpha):K]=2p$. Clearly $L_S=K(\alpha)$ and $L_S/K$ is Galois. As $\theta$ and $\phi$ are roots of the irreducible polynomial $f(t^p)=t^{2p} + xt^p +y$ over $K$ (As $f(t^p)=(t-\theta)^p(t-\phi)^p$), it is enough to show that $\phi\not \in K(\theta)$ for proving that $L/K$ is not normal. Assume on the contrary that $K(\theta)=K(\phi)$. Now let $u=-(\theta+\phi)$ and $v=\theta\phi$. Observe that $u^p=-(\theta^p+\phi^p)=-(\alpha+\beta)=x$ and $v^p=\theta^p\phi^p=\alpha\beta=y$. Thus $[K(u,v):K]=p^2$. But Since $u,v\in K(\theta)$, we have $p^2\mid 2p$ which is a contradiction. \smallskip

    Now we present some further observations about $L/K$. We have $r_K(L)=1, s_K(L)=2, l_K(L)=p, r_K(L_S)=2, s_K(L_S)=1$. We claim that $L_I=K$. As $[L:K]=2p$ with $p>2$, either $[L_I:K]=p$ or $L_I=K$. If $[L_I:K]=p$ then $[L_SL_I:K]=2p$. So $L=L_SL_I/K$ is normal which is a contradiction. Thus $a_K(L)=t_K(L)=2$ as maximal normal subextension $F'=L_SL_I=L_S$. Also $\tilde{L}=K(\theta, \phi)$ and $[\tilde{L}:K]=2p^2$. Let $A'=Aut(\tilde{L}/K)$. So $r_K(\tilde{L})=2=|A'|$. Let $1\neq \sigma\in A$ which maps $\sigma(\theta)=\phi$ and $\sigma(\phi)=\theta$. Thus $\sigma$ fixes $-(\theta+\phi)=u$ and $\theta\phi=v$. As $[K(u,v):K]=p^2$, we have $[\tilde{L}:K(u,v)]=2$. So $\tilde{L}^{A'}=(\tilde{L})_I=K(u,v)$. Thus $L_I\subsetneq (\tilde{L})_I$. Thus also $L_I L'_I\subsetneq (LL')_I$ where $L'=K(\phi)$. Also $(\tilde{L})_S=\tilde{L_S}=L_S=K(\alpha)$.

\end{example}

\subsection{Multiroot Capacity, Intersection Normal Indicium and Compositum Normal Indicium}\hfill

\smallskip

For a perfect base field, the notion of root capacity was defined by the author and Bhagwat in Definition 6.2.1 in \cite{Bhagwat_2025} which we can similarly define for any separable extension over any base field.

\begin{definition}
    \label{root capacity}
Let $L/K$ be a separable extension. By primitive element theorem $L=K(\alpha)$ for some $\alpha\in\bar{K}$. Let $f$ be minimal polynomial of $\alpha$ over $K$. For an extension $M/K$, root capacity of $M$ with respect to $L$ (with base field $K$ fixed) $\rho_K(M,L)$ is the number of roots of $f$ that are contained in $M$ (This is well defined by Proposition 6.2.2 in \cite{Bhagwat_2025}).  \smallskip

Equivalently by Proposition 6.2.6 (1) in \cite{Bhagwat_2025}, $\rho_K(M,L)=a\cdot r_K(L)$ where $a$ is number of distinct fields inside $M$ isomorphic to $L$ over $K$.
\end{definition}

We now generalize the above notion for any extension.

\begin{definition}
     Let $L/K$ be any extension. For an extension $M/K$, we define the root capacity of $M$ with respect to $L$ (with base field $K$ fixed) as $\rho_K(M,L):= a\cdot r_K(L)$ where $a$ is number of distinct fields inside $M$ isomorphic to $L$ over $K$. \smallskip

     For a simple extension $L=K(\alpha)$ over $K$ with $f$ being the minimal polynomial of $\alpha$ over $K$, $\rho_K(M,K(\alpha))$ counts the number of distinct roots of $f$ that are contained in $M$, which we denote by $\rho_K(M,f)$ and call root capacity of $M$ with respect to $f$.
\end{definition}

We also define the following for any extension.

\begin{definition}
     Let $L/K$ be any extension. For an extension $M/K$, we define the multiroot capacity of $M$ with respect to $L$ (with base field $K$ fixed) as $\lambda_K(M,L):= a\cdot l_K(L)$ where $a$ is number of distinct fields inside $M$ isomorphic to $L$ over $K$. \smallskip

     For a simple extension $L=K(\alpha)$ over $K$ with $f$ being the minimal polynomial of $\alpha$ over $K$, $\lambda_K(M,K(\alpha))$ counts the number of roots of $f$ with multiplicity that are contained in $M$, which we denote by $\lambda_K(M,f)$ and call multiroot capacity of $M$ with respect to $f$.
\end{definition}

\begin{proposition}\label{multi root cap prop}
Consider $L/K$ and $M/K$.

\begin{enumerate}
    \item $\lambda_K(M,L)=[L:L_S]\cdot \rho_K(M,L)$.\smallskip

    \item $\rho_K(L,L)=r_K(L)$ and $\lambda_K(L,L)=l_K(L)$ and $\rho_K(\tilde{L},L)=s_K(L)\cdot r_K(L)=[L_S:K]$ and $\lambda_K(\tilde{L},L)=[L:K]$.

    \smallskip

    \item $\lambda_K(L,L_S)=\rho_K(L,L_S)=r_K(L_S)$.\smallskip

    \item $\rho_K(L,L_I)=r_K(L_I)=1$ and $\lambda_K(L,L_I)=l_K(L_I)=[L_I:K]$.

\end{enumerate}
\end{proposition}

\begin{proposition}
Consider $L/K$. Consider the compositum $M$ of all distinct fields $L'/K$ that are isomorphic to $L/K$ such that $L'_S=L_S$. Then $\rho_K(M,L)=r_K(L_S)=\rho_K(M,L_S)$ 
\end{proposition}

\begin{proof}

 By Proposition \ref{s/s'}, the number of distinct extensions $L'/K$ that are isomorphic to $L/K$ such that $L'_S=L_S$ is exactly $s_K(L)/s_K(L_S)$ which we denote by $k$. Let $L_1,\dots, L_k$ be those fields. Thus $M=L_1\cdots L_k$. Suppose $L''/K$ be a field isomorphic to $L/K$ such that $L''\subset M$. Thus $(L'')_S\subset M_S=(L_1\cdots L_k)_S=(L_1)_S\cdots (L_k)_S=L_S$ by Lemma \ref{compositum} (1). Hence $(L'')_S=L_S$. Thus $L''=L_i$ for some $1\leq i\leq k$. Hence $L_1,\dots, L_k$ are all the distinct fields that are isomorphic to $L$ over $K$ and contained in $M$. Thus $\rho_K(M,L)=k\cdot r_K(L)=(s_K(L)/s_K(L_S))\cdot r_K(L)=r_K(L_S)$. Since $M_S=L_S$, the only field that is isomorphic to $L_S$ over $K$ and contained in $M$ is $L_S$ itself. Thus $\rho_K(M,L_S)=1\cdot r_K(L_S)=r_K(L_S)$.    
\end{proof}

For a perfect base field, the notion of intersection indicium was defined by the author in Definition 2.13 in \cite{jaiswal2025rootcapacityintersectionindicium} as a generalization of ascending index. We can similarly define, for any extension over any base field, intersection indicium as a generalization of ascending index and intersection normal indicium as a generalization of ascending normal index in light of Proposition \ref{f'f}.

\begin{definition}

 Let $L/K$ be any extension. For an extension $M/K$, let $L_1,L_2,\dots , L_a$ be all the distinct fields inside $M$ isomorphic to $L$ over $K$. Let $P=\cap_{i=1}^a (L_i)_S$ and $P'=\cap_{i=1}^a L_i$.
 \smallskip
 
 We define the intersection indicium of $M$ with respect to $L$ (with base field $K$ fixed) as\\
 $\tau_K(M,L):= [P:K]$. We also define the intersection normal indicium of $M$ with respect to $L$ (with base field $K$ fixed) as $\alpha_K(M,L):= [P':K]$. If none of the fields isomorphic to $L/K$ is contained in $M$, then we define $\tau_K(M,L)=0=\alpha_K(M,L)$.
\end{definition}

\begin{proposition} \label{int norm ind prop}Consider $L/K$ and $M/K$.

\begin{enumerate}

 \item $\tau_K(L,L)=[L_S:K]=\tau_K(L,L_S)$ and $\alpha_K(L,L)=[L:K]$ and $\tau_K(\tilde{L},L)=t_K(L)=\tau_K(\tilde{L}, L_S)$ and $\alpha_K(\tilde{L}, L)=a_K(L)$.\smallskip
 
 For any extension $M/K$, we have $t_K(L)\ |\ \tau_K(M,L)$ and $a_K(L)\mid \alpha_K(M,L)$. If $\tau_K(M,L)\neq 0\neq \alpha_K(M,L)$ then $\tau_K(M,L)\ |\ [L_S:K]$ and $\alpha_K(M,L)\mid [L:K]$. \smallskip

\item $(\tau_K(M,L)\cdot  i_K(L))\mid \alpha_K(M,L)\mid (\tau_K(M,L)\cdot [L:L_S])$. \smallskip

\item $L=L_SL_I$ $\iff$ for every $M/K$ we have $\alpha_K(M,L)=\tau_K(M,L)\cdot  i_K(L)$.

\smallskip

    \item $\alpha_K(M,L_S)=\tau_K(M,L_S)\mid \tau_K(M,L)$.\smallskip

\item  $M/K$ is normal $\implies$ $\alpha_K(M,L)=\tau_K(M,L)\cdot  i_K(L)$.
\end{enumerate}
    
\end{proposition}

\begin{proof}\hfill

\begin{enumerate}
\item Follows from Proposition \ref{f'f} (1), (3).\smallskip

\item  By Lemma \ref{compositum} (1), we have $P=\cap_{i=1}^a (L_i)_S=(\cap_{i=1}^a\ L_i)_S=P'_S$. By Proposition \ref{f'f} (6), $L_I\subset  P'$. As $PL_I\subset P'$, we have $(\tau_K(M,L)\cdot  i_K(L))\mid \alpha_K(M,L)$. Now $L_S/P$ is separable and $P'/P$ is purely inseparable. Thus $[P':K]/[P:K]=[P':P]=[P'L_S:L_S]\mid [L:L_S]$. Hence $\alpha_K(M,L)\mid (\tau_K(M,L)\cdot [L:L_S])$.
\smallskip

\item If $L=L_SL_I$, we have $[L:L_S]=i_K(L)$. Hence by part (2), $\alpha_K(M,L)=\tau_K(M,L)\cdot  i_K(L)$. Conversely, suppose for every $M/K$ we have $\alpha_K(M,L)=\tau_K(M,L)\cdot  i_K(L)$. Let $M=L$. By part (1), $[L:K]=\alpha_K(L,L)=\tau_K(L,L)\cdot i_K(L)=[L_S:K]\cdot [L_I:K]=[L_SL_I:K]$. Thus $L=L_SL_I$.\smallskip

\item Let $L_1,L_2,\dots , L_a$ be all the distinct fields inside $M$ isomorphic to $L$ over $K$ and let \\
$(L_S)_1,(L_S)_2,\dots , (L_S)_b$ be all the distinct fields inside $M$ isomorphic to $L_S$ over $K$. Clearly $\{(L_1)_S,(L_2)_S,\dots , (L_a)_S\}\subset \{(L_S)_1,(L_S)_2,\dots , (L_S)_b\}$ and the result follows.\smallskip

\item  If $M/K$ is normal then $\alpha_K(M,L)=a_K(L)=t_K(L)\cdot i_K(L)=\tau_K(M,L)\cdot  i_K(L)$.
    \end{enumerate}
\end{proof}



For a perfect base field, the notion of compositum indicium was defined by the author in Definition 2.16 in \cite{jaiswal2025rootcapacityintersectionindicium} as a generalization of degree of extension. We can similarly define, for any extension over any base field, compositum indicium as a generalization of degree of maximal separable subextension of an extension and compositum normal indicium as a generalization of degree of extension in light of Proposition \ref{f'f}. 

\begin{definition}

 Let $L/K$ be any extension. For an extension $M/K$, let $L_1,L_2,\dots , L_a$ be all the distinct fields inside $M$ isomorphic to $L$ over $K$.\smallskip
 
 We define the compositum indicium of $M$ with respect to $L$ (with base field $K$ fixed) as\\
 $\gamma_K(M,L):= [\Pi_{i=1}^a (L_i)_S:K]$. We also define the compositum normal indicium of $M$ with respect to $L$ (with base field $K$ fixed) as $\Gamma_K(M,L):= [\Pi_{i=1}^a L_i:K]$. If none of the fields isomorphic to $L/K$ is contained in $M$, then we define $\gamma_K(M,L)=0=\Gamma_K(M,L)$.\smallskip

 We also define the quantity $\iota_K(M,L) :=[(\Pi_{i=1}^a\ L_i)_I :K]$. If none of the fields isomorphic to $L/K$ is contained in $M$, then $\iota_K(M,L)=0$. 
\end{definition}

\begin{proposition} Consider $L/K$ and $M/K$.

\begin{enumerate}

 \item $\gamma_K(L,L)=[L_S:K]=\gamma_K(L,L_S)$ and $\Gamma_K(L,L)=[L:K]$ and $\iota_K(L,L)=i_K(L)$ and $\gamma_K(\tilde{L},L)=[(\tilde{L})_S:K]=[\tilde{(L_S)}:K]=\gamma_K(\tilde{L}, L_S)$ and $\Gamma_K(\tilde{L}, L)=[\tilde{L}:K]$ and $\iota_K(\tilde{L},L)=[(\tilde{L})_I:K]=i_K(\tilde{L})$. 

 \smallskip

 For any extension $M/K$, $[L_S:K]\ |\ \gamma_K(M,L)$ and $[L:K]\mid \Gamma_K(M,L)$ and $i_K(L)\mid \iota_K(M,L)$. If $\gamma_K(M,L)\neq 0\neq \Gamma_K(M,L)$ then $\gamma_K(M,L)\ |\ [(\tilde{L})_S:K]$ and $\Gamma_K(M,L)\mid [\tilde{L}:K]$. If $\iota_K(M,L)\neq 0$ then $\iota_K(M,L)\mid i_K(\tilde{L})$. \smallskip

\item $(\gamma_K(M,L)\cdot  \iota_K(M, L))\mid \Gamma_K(M,L)\mid (\gamma_K(M,L)\cdot i_K(\tilde{L}))$. \smallskip

\item $L=L_SL_I$ $\iff$ for every $M/K$ we have $\Gamma_K(M,L)=\gamma_K(M,L)\cdot  i_K(L)$.

\smallskip

    \item $\gamma_K(M,L)\mid \gamma_K(M,L_S)= \Gamma_K(M,L_S)$.
\smallskip

\item $M/K$ is normal $\implies$ $\gamma_K(M,L)\cdot \iota_K(M,L)=\Gamma_K(M,L)$.
\end{enumerate}
    
\end{proposition}

\begin{proof} Let $L_M=\Pi_{i=1}^a\ L_i$. 

\begin{enumerate}
\item Follows from Proposition \ref{f'f} (1), (3), (4). By Proposition \ref{f'f} (6), $L_I\subset (L_M)_I \subset (\tilde{L})_I$.

\smallskip

\item  By Lemma \ref{compositum} (1), we have $(L_M)_S=(\Pi_{i=1}^a\ L_i)_S=\Pi_{i=1}^a\ (L_i)_S$. As $(L_M)_S (L_M)_I \subset L_M$, we have $(\gamma_K(M,L)\cdot  \iota_K(M,L))\mid \Gamma_K(M,L)$. Now $(\tilde{L})_S/(L_M)_S$ is separable and $L_M/(L_M)_S$ is purely inseparable. Thus $[L_M:K]/[(L_M)_S:K]=[L_M:(L_M)_S]=[L_M(\tilde{L})_S:(\tilde{L})_S]\mid [\tilde{L}:(\tilde{L})_S]=[(\tilde{L})_I:K]=i_K(\tilde{L})$. Hence $\Gamma_K(M,L)\mid (\gamma_K(M,L)\cdot i_K(\tilde{L}))$. \smallskip

\item If $L=L_SL_I$, we have $(\tilde{L})_I=L_I$ by Proposition \ref{f'f} (7) and so $i_K(L)=\iota_K(M,L)=i_K(\tilde{L})$. Hence by part (2), $\Gamma_K(M,L)=\gamma_K(M,L)\cdot  i_K(L)$. Conversely, suppose for every $M/K$ we have $\Gamma_K(M,L)=\gamma_K(M,L)\cdot  i_K(L)$. Let $M=L$. By part (1), $[L:K]=\Gamma_K(L,L)=\gamma_K(L,L)\cdot i_K(L)=[L_S:K]\cdot [L_I:K]=[L_SL_I:K]$. Thus $L=L_SL_I$.\smallskip

\item Follows from proof of Proposition \ref{int norm ind prop} (4).

\smallskip

\item If $M/K$ is normal then $\Gamma_K(M,L)=[\tilde{L}:K]=[(\tilde{L})_S:K]\cdot [(\tilde{L})_I :K]=\gamma_K(M,L)\cdot  \iota_K(M,L)$.

    \end{enumerate}
\end{proof}


\section{General Normal Magnification}\label{GNM Section}
   
The notion of general magnification for separable extensions was introduced by the author and others in \cite{jaiswal2025variantsinverseclustersize}. We generalize this notion further for any finite extension.

\subsection{General Normal Magnification}

\begin{definition}\label{GNM}
    A finite extension $M/K$ is said to be obtained by general normal magnification from a subextension $L/K$ if we have the following:
\smallskip
 
 \begin{enumerate}
 \item  $[L:K] > 1,$ 
 \smallskip
      
\item there exists a finite extension $J/K$ such that the normal closure $\tilde{L}$ of $L$ in $\bar{K}$ and normal closure $\tilde{J}$ of $J$ in $\bar{K}$ are linearly disjoint over $K$. \smallskip

\item $LJ=M$.\smallskip

 \end{enumerate}

The magnification is called trivial if $J = K$ and nontrivial otherwise. 
 \end{definition}

 \begin{remark}
     Definition \ref{GNM} is symmetric in the nontrivial case. Suppose $M/K$ is obtained by nontrivial general normal magnification from $L/K$ through $J/K$. Then $M/K$ is obtained by nontrivial general normal magnification from $J/K$ through $L/K$.
 \end{remark}

 \begin{proposition}
\label{gen norm for L}     
 
Consider $L/K$.
 \begin{enumerate}
     \item 
   For any general $L/K$ if $[L_S:K]>1$, we have that $L_SL_I/K$ is obtained by general normal magnification from $L_S/K$ through $L_I/K$. 
     
     \smallskip
     
\item Suppose $L/K$ is normal. Then if $[L_S:K]>1$ then we have that $L/K$ is obtained by general normal magnification from $L_S/K$ through $L_I/K$. 

     \smallskip

\item  For any general $L/K$ with $[L_S:K]>1$, we have that $\tilde{L}/K$ is obtained by general magnification from $(\tilde{L})_S/K$ through $(\tilde{L})_I/K$.

     \smallskip

  \item  With notations as in Proposition \ref{f'f}, if $t_K(L)>1$ then, $F'/K$ is obtained by general normal magnification from $F/K$ through $L_I/K$.
 
     \end{enumerate}
 \end{proposition}

 \begin{proof}
     \hfill
     \begin{enumerate}
     \item Now $L_I/K$ is normal. By Example 20.13 in \cite{morandi2012field}, $\tilde{(L_S)}$ and $L_I$ are linearly disjoint over $K$.

     \smallskip
     
     \item Now as $L/K$ is normal, $L_S/K$ is Galois. By Corollary \ref{LI cor}, we have $L=L_S L_I$. 
         
         \smallskip

    \item By Proposition \ref{f'f} (4), $(\tilde{L})_S=\tilde{(L_S)}$. So result follows as $[L_S:K]>1\iff [\tilde{(L_S)}:K]>1$.

    \smallskip

         \item By Proposition \ref{ad prop} (1) we have $F'=FL_I$.

     \end{enumerate}
 \end{proof}

\begin{corollary}

    Consider $L/K$ with $L_S/K$ being a nontrivial Galois extension. We have $L/K$ is obtained by general normal magnification from $L_S/K$ through some $J/K$ $\iff$ $L/K$ is normal. 
\end{corollary}

\begin{proof}
    One implication follows from Proposition \ref{gen norm for L} (2) with $J=L_I$. Conversely suppose $L/K$ is obtained by general normal magnification from $L_S/K$ through some $J/K$. Then $J_S=K$ otherwise $\tilde{L_S}$ and $\tilde{J}$ will not be linearly disjoint over $K$. Thus $J/K$ is purely inseparable and so $J\subset L_I$. As $L=L_SJ$, we have $L=L_SL_I$. Also as $L_S/K$ is Galois, $L/K$ is normal.
\end{proof}

The following follows from Proposition \ref{L'S prop}.

\begin{proposition} With notations as in Proposition \ref{L'S prop}, we have the following:\begin{enumerate}
\item  If $s_K(L)>1$ then we have that $L'_S/K$ is obtained by general normal magnification from $(L^A)_S/K$ through $L_I/K$. 
\smallskip

\item Suppose $r_K(L)>1$. Then\begin{enumerate}
\item $L/(L^A)_S$ is obtained by general normal magnification from $L_S/(L^A)_S$ through $L^A/(L^A)_S$.

\item $L_SL_I/(L^A)_S$ is obtained by general normal magnification from $L_S/(L^A)_S$ through the subextension $L'_S/(L^A)_S$.

\item  $L/L'_S$ is obtained by general normal magnification from $L_SL_I/L'_S$ through $L^A/L'_S$. 

\end{enumerate}

  \end{enumerate}
\end{proposition}

\begin{lemma}\label{lin disj lem}
    Consider extensions $L/K$ and $J/K$ such that $L$ and $J$ are linearly disjoint over $K$. Let $L_1/K, L_2/K$ be subextensions of $L/K$ and $J_1/K,J_2/K$ be subextensions of $J/K$. Then we have $L_1 J_1 \cap L_2 J_2= (L_1\cap L_2)(J_1\cap J_2)$.
\end{lemma}

\begin{proof}
   Clearly $(L_1\cap L_2)(J_1\cap J_2)\subset L_1 J_1 \cap L_2 J_2$. Conversely, suppose $\delta \in L_1 J_1 \cap L_2 J_2$. Then $\delta=\Sigma_{i=1}^{[J_1:K]}\ l_ij_i=\Sigma_{k=1}^{[J_2:K]}\ l'_kj'_k$ where $\{j_i\}_{i=1}^{[J_1:K]}$ is a basis for $J_1/K$ and $\{j'_k\}_{k=1}^{[J_2:K]}$ is a basis for $J_2/K$ and all $l_i\in L_1$ and all $l'_k\in L_2$. Now we can extend this basis of $J_1/K$ to a basis $B=\{j_i\}_{i=1}^{[J_1:K]}\cup \{j''_i\}_{i=[J_1:K]+1}^{[J:K]}$ of $J/K$. By linear disjointness condition, $B$ is also a basis of $LJ/L$. Also we can write each $j'_k\in J$ in terms of basis $B$ as $j'_k=\Sigma_i\ k'_{ki} j_i + \Sigma_i\ k''_{ki} j''_i$ where all $k'_{ki}, k''_{ki}\in K$. Thus $\delta=\Sigma_k\ l'_k (\Sigma_i\ k'_{ki} j_i + \Sigma_i\ k''_{ki} j''_i)=\Sigma_i\ (\Sigma_k\ k'_{ki} l'_k) j_i + \Sigma_i\ (\Sigma_k\ k''_{ki} l'_k) j''_i=\Sigma_i\ l_i j_i$. As $B$ is a basis of $LJ/L$, we have that each $l_i=\Sigma_k\ k'_{ki} l'_k$. Thus each $l_i\in L_2$ so each $l_i\in L_1\cap L_2$. Thus $\delta\in (L_1\cap L_2) J_1$. Similarly we have $\delta\in (L_1\cap L_2)J_2$. Thus $\delta \in (L_1\cap L_2)J_1 \cap (L_1\cap L_2)J_2$. Now choosing basis of $(L_1\cap L_2)/K$ and repeating a similar argument as before, we have $\delta\in (L_1\cap L_2)(J_1\cap J_2)$. 
\end{proof}

\begin{remark}
 Suppose $M/K$ is obtained by general normal magnification from  $L/K$ through $J/K$. Consider $K\subset M'\subset M$. It is not necessary that $M'=L'J'$ for some $K\subset L'\subset L$ and $K\subset J'\subset J$. Following is an example demonstrating this.

\smallskip

Let $K=\mathbb{F}_p (x,y), L=K(x^{1/p}) , J=K(y^{1/p})$. So $M=LJ=K(x^{1/p}, y^{1/p})$. Now $L/K, J/K$ are  purely inseparable (hence normal) having degree $p$ and $M/K$ is purely inseparable of degree $p^2$. So $M/K$ is obtained by general normal magnification from $L/K$ through $J/K$. Let $M'=K (x^{1/p}+y^{1/p})$. As $x^{1/p}+y^{1/p}=(x+y)^{1/p}$, so $M'/K$ is also purely inseparable of degree $p$. As $M'\neq K,L,J,LJ$, so we are done.\smallskip

This example also demonstrates that: Suppose $M/K$ is obtained by general normal magnification from a subextension $L/K$ through $J/K$. Such a $J/K$ need not be unique. Observe that in the above example $M/K$ is also obtained by general normal magnification from subextension $L/K$ through $M'/K$ where $J\neq M'$.
\end{remark}

\begin{proposition}
    
\label{GNM prop}
    Suppose $M/K$ is obtained by general normal magnification from  $L/K$ through $J/K$. Let $\tilde{M}$ be normal closure of $M/K$ inside $\bar{K}$. Then 
\smallskip 

\begin{enumerate}

\smallskip  \item 
\begin{enumerate}
    \item $\tilde{M}/K$ is obtained by general normal magnification from $\tilde{L}/K$ through $\tilde{J}/K$.

     \item $M\tilde{L}$ and $\tilde{J}$ are linearly disjoint over $J$.

    \item $M\tilde{J}$ and $\tilde{L}$ are linearly disjoint over $L$.

    \item $M\tilde{L}$ and $ M\tilde{J}$ are linearly disjoint over $M$.

\end{enumerate}
\smallskip

 \item 
 \begin{enumerate}

 \item $M_S=L_SJ_S$.
 
\item If $[L_S:K]>1$, then $M_S/K$ is obtained by general normal magnification from $L_S/K$ through $J_S/K$.

\end{enumerate}

\smallskip

\item \begin{enumerate}
    \item $(\tilde{M})_S=(\tilde{L})_S(\tilde{J})_S$ and $(\tilde{M})_I=(\tilde{L})_I(\tilde{J})_I$.

    \item If $[L_S:K]>1$, then $(\tilde{M})_S/K$ is obtained by general normal magnification from $(\tilde{L})_S/K$ through $(\tilde{J})_S/K$.

    \item If $[(\tilde{L})_I:K]>1$, then $(\tilde{M})_I/K$ is obtained by general normal magnification from $(\tilde{L})_I/K$ through $(\tilde{J})_I/K$.
\end{enumerate}

\smallskip

\item \begin{enumerate}
    \item $M_I=L_IJ_I$.

    \item If $[L_I:K]>1$, then $M_I/K$ is obtained by general normal magnification from $L_I/K$ through $J_I/K$.
\end{enumerate}

    \end{enumerate}
\end{proposition}

\begin{proof}\hfill

\begin{enumerate}
    \item Part (a) is trivial. Parts (b) and (c) follow from Theorem 20.12 in \cite{morandi2012field}. Part (d) follows from Lemma 20.4 in \cite{morandi2012field} by observing that $[M\tilde{L}:M]\cdot[M\tilde{J}:M]=[\tilde{L}:L]\cdot[\tilde{J}:J]=\frac{[\tilde{L}:K]\cdot[\tilde{J}:K]}{[L:K]\cdot[J:K]}=\frac{[\tilde{M}:K]}{[M:K]}=[\tilde{M}:M]$. 
    
\smallskip

\item Part (a) follows from Lemma \ref{compositum} (1). Now consider $[L_S:K]>1$. Since $\tilde{L}$ and $\tilde{J}$ are linearly disjoint over $K$. Thus the Galois closures $\tilde{L_S}$ and $\tilde{J_S}$, of $L_S/K$ and $J_S/K$ respectively, are linearly disjoint over $K$.\smallskip

\item 
Now $\tilde{M}=\tilde{L}\tilde{J}$, so we have $(\tilde{M})_S=(\tilde{L})_S(\tilde{J})_S$ by Lemma \ref{compositum} (1). Since $\tilde{L}/K$ and $\tilde{J}/K$ are normal, we get by Lemma \ref{compositum} (2) $(\tilde{M})_I=(\tilde{L}\tilde{J})_I=(\tilde{L})_I(\tilde{J})_I$.

  \smallskip

\item 
Clearly $M_I=M\cap (\tilde{M})_I$. So $M_I=LJ\cap (\tilde{L})_I(\tilde{J})_I$. Since $\tilde{L}$ and $\tilde{J}$ are linearly disjoint over $K$, so by Lemma \ref{lin disj lem}, $M_I=(L\cap (\tilde{L})_I)(J\cap (\tilde{J})_I)=L_IJ_I$.

\end{enumerate}\end{proof}

We generalize Corollary 8.1.5 in \cite{Bhagwat_2025} for our case. 

\begin{lemma}\label{isom lem}
    Let $M/K$ be obtained by general normal magnification from $L/K$ through $J/K$. The following are equivalent
    \begin{enumerate}
        \item  an extension $M'/K$ is isomorphic to $M/K$.

        \item $M'=L'J'$ where $L'/K$ is isomorphic to $L/K$ and $J'/K$ is isomorphic to $J/K$. 
    \end{enumerate}

    Further in this case the extensions $L'/K$ and $J'/K$ are unique and are given by $L'=M'\cap \tilde{L}$ and $J'=M'\cap \tilde{J}$.
\end{lemma}

\begin{proof}
 Suppose $M'/K$ is isomorphic to $M/K$. Let $\sigma : LJ\rightarrow M'$ be the isomorphism such that $\sigma|_{K}=id_{K}$. Let $\sigma(L)=L'$ and $\sigma(J)=J'$. So (1)$\implies$(2).\smallskip 

    Conversely, suppose $M'=L'J'$ where $L'/K$ is isomorphic to $L/K$ and $J'/K$ is isomorphic to $J/K$. Let $\lambda : L\rightarrow L'$ be the isomorphism such that $\lambda|_{K}=id_{K}$ and $\nu : J\rightarrow J'$ be the isomorphism such that $\nu|_{K}=id_{K}$. Let $\sigma: M=LJ\rightarrow M'=L'J'$ be such that $\sigma(l)=\lambda(l)$ for all $l\in L$ and $\sigma(j)=\nu(j)$ for all $j\in J$. Let $\{ l_i \}_{1 \leq i \leq [L:K]}$ be a $K$-basis for $L$.  Hence any element of $LJ$ is of the form $\Sigma\ l_i j_i$ for $j_i\in J$.\smallskip

    Suppose $\Sigma\ l_i j_i=0$. Since $l_i\in L\subset \tilde{L}$ are linearly independent over $K$, and $\tilde{L}$ and $\tilde{J}$ are linearly disjoint over $K$; it follows by Definition 20.1 in \cite{morandi2012field} that $\{ l_i \}_{1 \leq i \leq [L:K]}$ are linearly independent over $\tilde{J}$. Thus $j_i=0$ for all $i$. Now $\sigma (\Sigma\ l_i j_i)= \Sigma\ \lambda (l_i) \nu(j_i)=0$. Hence $\sigma$ is well defined homomorphism with $\sigma|_K=id_K$.\smallskip

    Suppose $\Sigma\ l_ij_i\in LJ$ such that $\sigma(\Sigma\ l_ij_i)=0$. Thus $\Sigma\ \lambda(l_i)\nu(j_i)=0$. Since $\lambda$ is a $K$-isomorphism, we have that $\{\lambda(l_i)\}_{1\leq i \leq [L:K]}$ is a $K$-basis for $L'$. Since $\lambda(l_i)\in L'\subset \tilde{L}$ are linearly independent over $K$, and $\tilde{L}$ and $\tilde{J}$ are linearly disjoint over $K$; $\lambda(l_i)$ are linearly independent over $\tilde{J}$. Thus $\nu(j_i)=0$ for all $i$. Since $\nu$ is a $K$-isomorphism, we have $j_i=0$ for all $i$. This $\Sigma\ l_ij_i=0$. Hence $\sigma$ is injective. Also since $\lambda$ and $\nu$ are $K$-isomorphisms, we have that $\sigma$ is surjective. So $\sigma$ is the required $K$-isomorphism.\smallskip

    As $L'\subset \tilde{L}$ and $J'\subset \tilde{J}$, so by Theorem 20.12 in \cite{morandi2012field}, we have $\tilde{L}$ and $L'J'$ to be linearly disjoint over $L'$ and $\tilde{J}$ and $L'J'$ to be linearly disjoint over $J'$. Thus $\tilde{L}\cap L'J'=L'$ and $\tilde{J}\cap L'J'=J'$. Hence the uniqueness of $L'$ and $J'$ follows.
\end{proof}

\begin{corollary}\label{isom lem cor}

 Let $M/K$ be obtained by general normal magnification from $L/K$ through $J/K$. The following are equivalent
    \begin{enumerate}
        \item  an extension $M'/J$ is isomorphic to $M/J$.

        \item $M'=L'J$ where $L'/K$ is isomorphic to $L/K$.
    \end{enumerate}

    Further in this case the extensions $L'/K$ is unique and is given by $L'=M'\cap \tilde{L}$.
    
\end{corollary}

\begin{proof}
    Suppose $M'/J$ is isomorphic to $M/J$. Let $\sigma : LJ\rightarrow M'$ be the isomorphism such that $\sigma|_{J}=id_{J}$. Let $\sigma(L)=L'$. So (1)$\implies$(2).
\smallskip

Conversely, suppose $M'=L'J$ where $L'/K$ is isomorphic to $L/K$. Let $\lambda : L\rightarrow L'$ be the isomorphism such that $\lambda|_{K}=id_{K}$ and $\nu : J\rightarrow J$ be the identity isomorphism. Let $\sigma: M=LJ\rightarrow M'=L'J$ be such that $\sigma(l)=\lambda(l)$ for all $l\in L$ and $\sigma(j)=\nu(j)=j$ for all $j\in J$. By proof of Lemma \ref{isom lem}, $\sigma$ is the required $J$-isomorphism. Also $L'J\cap \tilde{L}=L'$.\end{proof}

The following results justifies the terminology in Definition \ref{GNM}. 

\begin{proposition}
\label{GNM theorem}
 
 (General Normal Magnification I) Let $M/K$ be obtained by general normal magnification from $L/K$ through $J/K$. Then
 \begin{enumerate}
     \item $[M:K]=  [L:K]\cdot [J:K]$ and $s_K(M)= s_K(L)\cdot s_K(J)$.
\smallskip
     
     \item  $r_K(M)= r_K(L)\cdot r_K(J)$ and $l_K(M)=l_K(L)\cdot l_K(J)$.
 \smallskip
 
     \item  $t_K(M)= t_K(L)\cdot t_K(J)$ and $u_K(M)= u_K(L)\cdot u_K(J)$.
     \smallskip

\item $i_K(M)=i_K(L)\cdot i_K(J)$ and $a_K(M)=a_K(L)\cdot a_K(J)$.

 \end{enumerate}

 \end{proposition}

\begin{proof}\hfill
\begin{enumerate}
    \item The first assertion is true by Lemma 20.4 in \cite{morandi2012field} and the second follows from Lemma \ref{isom lem}.

    \smallskip
    \item By Proposition \ref{GNM prop} (2) it follows that $[M_S:K]=[L_S:K]\cdot[J_S:K]$. By Proposition \ref{r prop} (3) and Part (1) we are done.

\smallskip
    \item By Proposition \ref{asc ind}, $t_K(M)=t_K(M_S)$ and $u_K(M)=u_K(M_S)$. If $L_S=K$, then by Proposition \ref{GNM prop} (2) we have $M_S=J_S$ and the result follows. Now if $[L_S:K]>1$, then by Proposition \ref{GNM prop} (2), $M_S/K$ is obtained by general normal magnification from $L_S/K$ through $J_S/K$. So we are done by Proposition 3.3 in \cite{jaiswal2025variantsinverseclustersize}.\smallskip

\item By Proposition \ref{GNM prop} (4) it follows that $[M_I:K]=[L_I:K][J_I:K]$. The other assertion follows from part (3) and Proposition \ref{ad prop} (1).

\end{enumerate}
 
\end{proof}

\begin{remark}
  We can now give an alternate proof for Proposition \ref{L'S prop} (4) as follows. Now $L=L_SL_I$. If $L_S=K$, then $L=L_I$. Thus $r_K(L)=1=r_K(L_S)$. If $[L_S:K]>1$ then $L/K$ is obtained by general normal magnification from $L_S/K$ through $L_I/K$. Thus by Proposition \ref{GNM theorem} (2), $r_K(L)=r_K(L_S)\cdot r_K(L_I)=r_K(L_S)$.
\end{remark}

The following is a generalization of Lemma 2.28 in \cite{jaiswal2025rootcapacityintersectionindicium}.

\begin{proposition}\label{GNM thm II}

(General Normal Magnification II) Consider $M'/L'/K$. Let $M'/K$ be obtained by general normal magnification from $M/K$ through $J'/K$ and $L'/K$ be obtained by general normal magnification from $L/K$ through $J/K$ where $L\subset M$ and $J\subset J'$. Then

    \begin{enumerate}
        \item $\rho_K(M',L')=\rho_K(M,L)\cdot \rho_K(J',J)$ and $\lambda_K(M',L')=\lambda_K(M,L)\cdot \lambda_K(J',J)$.
\smallskip

        \item $\tau_K(M',L')=\tau_K(M,L)\cdot \tau_K(J',J)$ and $\alpha_K(M',L')=\alpha_K(M,L)\cdot \alpha_K(J',J)$.

        \smallskip

\item $\gamma_K(M',L')=\gamma_K(M,L)\cdot \gamma_K(J',J)$ and $\Gamma_K(M',L')=\Gamma_K(M,L)\cdot \Gamma_K(J',J)$ and $\iota_K(M',L')=\iota_K(M,L)\cdot \iota_K(J',J)$.
    \end{enumerate}
\end{proposition}

\begin{proof} By Proposition \ref{GNM theorem} (1), $s_K(L')=s_K(L)\cdot s_K(J)$. Thus the $s_K(L')$ many distinct fields isomorphic to $L'/K$ are precisely $L_iJ_j$ for $1\leq i\leq s_K(L)$ and $1\leq j\leq s_K(J)$ where $L_i$'s are all the distinct fields isomorphic to $L/K$ and $J_j$'s are all the distinct fields isomorphic to $J/K$. Now suppose $L_i J_j\subset M'$. Thus $J_j\subset M'$. Hence $J_j\cap \tilde{J'}\subset M'\cap \tilde{J'}$. By Proposition \ref{GNM prop} (1)(b), $M'\tilde{M}\cap \tilde{J'}=J'$. So $M'\cap \tilde{J'}=J'$. Since $J_j\subset \tilde{J'}$ we have $J_j\subset J'$. Similarly $L_i\subset M$. Thus the distinct fields inside $M'$ isomorphic to $L'/K$ are precisely $\{L_iJ_j\}_{1\leq i\leq a, 1\leq j\leq b}$ where $\{L_i\}_{i=1}^a$ are fields inside $M$ isomorphic to $L/K$ and $\{J_j\}_{j=1}^b$ are fields inside $J'$ isomorphic to $J/K$. 

\begin{enumerate}
\item  Now $\rho_K(M',L')/r_K(L')=(\rho_K(M,L)/r_K(L))\cdot (\rho_K(J',J)/r_K(J))=(\lambda_K(M,L)/l_K(L))\cdot (\lambda_K(J',J)/l_K(J))=\lambda_K(M',L')/l_K(L')$. By Proposition \ref{GNM theorem} (2), we are done.

\smallskip

\item 

By Lemma \ref{lin disj lem}, $(\cap_{1\leq i\leq a, 1\leq j\leq b}\ L_iJ_j)=(\cap_{i=1}^a L_i)(\cap_{j=1}^b J_j)$. Also by Lemma \ref{compositum} (1),\\
$(\cap_{1\leq i\leq a, 1\leq j\leq b}\ (L_iJ_j)_S)=(\cap_{1\leq i\leq a, 1\leq j\leq b}\ (L_i)_S(J_j)_S)=(\cap_{i=1}^a (L_i)_S)(\cap_{j=1}^b (J_j)_S)$.

\smallskip

\item 

Clearly, $\Pi_{1\leq i\leq a, 1\leq j\leq b}\ L_i J_j=(\Pi_{i=1}^a\ L_i)(\Pi_{j=1}^b\ J_j)$. Also by Lemma \ref{compositum} (1),\\
$\Pi_{1\leq i\leq a, 1\leq j\leq b}\ (L_i J_j)_S=\Pi_{1\leq i\leq a, 1\leq j\leq b}\ (L_i)_S (J_j)_S=(\Pi_{i=1}^a\ (L_i)_S)(\Pi_{j=1}^b\ (J_j)_S)$. Also by proof of Proposition \ref{GNM prop} (4), $(\Pi_{1\leq i\leq a, 1\leq j\leq b}\ L_i J_j)_I=((\Pi_{i=1}^a\ L_i)(\Pi_{j=1}^b\ J_j))_I=(\Pi_{i=1}^a\ L_i)_I(\Pi_{j=1}^b\ J_j)_I$.
    
\end{enumerate}
    
\end{proof}

We have the following hereditary property for general normal magnification.

\begin{proposition} \label{hereditary}
    Suppose $M/K$ is obtained by general normal magnification from  $L/K$ through $J/K$ as in Def \ref{GNM}. Let $K\subset L' \subset L$ and $K\subset J'\subset J$ with $[L:L']>1$. Then $M/L'J'$ is obtained by general normal magnification from $LJ'/L'J'$ through $L'J/L'J'$.
    
\end{proposition}

\begin{proof}
 We check that the conditions in Definition \ref{GNM} hold. Let $L_1$ be normal closure of $LJ'/L'J'$ and $J_1$ be normal closure of $L'J/L'J'$. 
 
 \smallskip

Since $\tilde{L}$ and $\tilde{J}$ are linearly disjoint over $K$. Hence $\tilde{L}$ and $J'$ are linearly disjoint over $K$. Thus $\tilde{L}$ and $ L'J'$ are linearly disjoint over $L'$. Thus $[LJ':L'J']=[L:L']>1$. Also since $\tilde{L}/L$ is normal, we have $\tilde{L}J'/L'J'$ to be normal. Thus $L_1\subset \tilde{L}J'$. Similarly we have $J_1\subset L'\tilde{J}$.

\smallskip

Now $(\tilde{L}J')(L'\tilde{J})=\tilde{L}\tilde{J}=\tilde{M}$. By similar computation as in Proposition \ref{GNM prop} (1), we can show that $[\tilde{L}J':L'J']\cdot [L'\tilde{J}:L'J']=[\tilde{M}:L'J']$. Thus we have $\tilde{L}J'$ and $L'\tilde{J}$ are linearly disjoint over $L'J'$. Hence $L_1$ and $J_1$ are linearly disjoint over $L'J'$. Also $(LJ')(L'J)=LJ=M$. \end{proof}

By letting $J'=K$, we have the following.

\begin{corollary}
    Let $M/K$ be obtained by general normal magnification from $L/K$ through $J/K$. Then for any $K\subset L'\subset L$ with $[L:L']>1$, the extension $M/L'$ is obtained by general normal magnification from $L/L'$ through $L'J/L'$.
\end{corollary}

\subsection{Unique Chains and General normal magnification}

We have the notions of unique descending chains and unique ascending chains for extensions over perfect base field introduced in Section 7 in \cite{Bhagwat_2025}. We can similarly define the notions for any extension over any base field. Let $L/K$ be a nontrivial finite extension. \smallskip

In light of Proposition \ref{r prop} (1) we have a unique strictly descending chain of subextensions \[ L=N_0\supsetneq N_1 \supsetneq N_2 \supsetneq \dots \supsetneq N_k\]
        such that for all $1\leq i\leq k$, $N_i$ is the unique intermediate extension for $N_{i-1}/K$ such that $N_{i-1}/N_i$ is Galois of maximum possible degree. 
\smallskip

 In light of Definition \ref{asc ind def} we have a unique strictly ascending chain of subextensions  \[ K=F_0\subsetneq F_1 \subsetneq F_2 \subsetneq \dots \subsetneq F_l \] such that for all $1\leq j\leq l$, $F_j$ is the unique intermediate extension for $L/F_{j-1}$ such that $F_j/F_{j-1}$ is Galois of maximum possible degree.\smallskip

        Both the unique chains terminate since $L/K$ is finite. Also $r_K(N_k)=1$ and $t_{F_l}(L)=1$.

\begin{proposition}\label{LSNi}
 Consider the unique descending chain for $L/K$ as above. 
 \begin{enumerate}
\item If $r_K(L)>1$, then for all $1\leq i\leq k$, $L/(N_i)_S$ is obtained by general normal magnification from $L_S/(N_i)_S$ through $N_i/(N_i)_S$. \smallskip

\item For all $1\leq i\leq k$, $L_I=(N_i)_I$.

 \end{enumerate}
 
\end{proposition}

\begin{proof}\hfill

\begin{enumerate}
\item    We know that $L=L_SL^{Aut(L/K)}=L_S N_1$. Similarly we can write $N_1=(N_1)_S N_2$. Thus $L=L_S((N_1)_S N_2)=L_S N_2$ as $(N_1)_S\subset L_S$. Proceeding similarly we get $L=L_SN_i$ for all $i$. Also $L_S/(N_i)_S$ is separable and $N_i/(N_i)_S$ is purely inseparable and $[L_S:(N_i)_S]\geq [L_S:(N_1)_S]=r_K(L)>1$. \smallskip

\item Since $L_I\subset L^A=N_1$ where $A=Aut(L/K)$. So $L_I=(N_1)_I$. Proceeding similarly we are done.

\end{enumerate}
\end{proof}

\begin{proposition}\label{same uniq asc}
    The unique ascending chain for $L/K$ is the unique ascending chain for $L_S/K$.
\end{proposition}

\begin{proof}
   Consider the unique ascending chain for $L/K$. Clearly $F_1\subset L_S$. Suppose $M/F_1$ is the maximal separable subextension of $L/F_1$. Since $L_S/F_1$ is separable thus $L_S\subset M$. Also $M/L_S$ is separable. But since $L/L_S$ is purely inseparable we have $M/L_S$ to be purely inseparable as well. Thus $M=L_S$. Hence $F_2\subset L_S$. Proceeding in the same way we are done.
\end{proof}

\begin{lemma}\label{aut lem}
 Let $M/K$ be obtained by general normal magnification from $L/K$ through $J/K$. Then $Aut(M/K)\cong Aut(L/K)\times Aut(J/K)$.    
\end{lemma}

\begin{proof}
    Suppose $\sigma\in Aut(M/K)$. Thus $\sigma:LJ\rightarrow LJ$ is a $K$-isomorphism. Let $\sigma(L)=L'$ where $L'/K$ is isomorphic to $L/K$. Since $L'\subset LJ$. Thus $L'\cap \tilde{L}\subset LJ\cap \tilde{L}$. Thus $L'\subset L$ so $\sigma(L)=L'=L$. Similarly we have $\sigma(J)=J$. Thus we have a well defined group homomorphism $\Phi : Aut(M/K)\rightarrow Aut(L/K)\times Aut(J/K)$ where $\Phi(\sigma)=(\sigma|_L,\sigma|_J)$. Now $\Phi$ is clearly injective. The surjectivity follows from the proof of Lemma \ref{isom lem}.
\end{proof}

Now we generalize Theorem 3.8 in \cite{jaiswal2025variantsinverseclustersize}.
   
\begin{proposition}
    
\label{unique chain GNM}

   Suppose the extension $M/K$ is obtained by general normal magnification from  $L/K$ through $J/K$. 
     
     \begin{enumerate}
\item Let the unique descending chain for $L/K$ be $L=N_0\supsetneq N_1 \supsetneq N_2 \supsetneq \dots \supsetneq N_k$ and the unique descending chain for $J/K$ be $J=N'_0\supsetneq N'_1 \supsetneq N'_2 \supsetneq \dots \supsetneq N'_{k'}$. Then the unique descending chain for $M/K$ is $M=N_0 N'_0\supsetneq N_1 N'_1 \supsetneq N_2 N'_2 \supsetneq \dots \supsetneq N_{k'} N'_{k'}\supsetneq N_{k'+1} N'_{k'}\supsetneq \dots  \supsetneq N_{k} N'_{k'}$ for $k\geq k'$ and $M=N_0 N'_0\supsetneq N_1 N'_1 \supsetneq N_2 N'_2 \supsetneq \dots \supsetneq N_{k} N'_{k}\supsetneq N_{k} N'_{k+1}\supsetneq \dots  \supsetneq N_{k} N'_{k'}$ for $k<k'$.

\smallskip

\item Let the unique ascending chain for $L/K$ be $K=F_0\subsetneq F_1 \subsetneq F_2 \subsetneq \dots \subsetneq F_l$ and the unique ascending chain for $J/K$ be $K=F'_0\subsetneq F'_1 \subsetneq F'_2 \subsetneq \dots \subsetneq F'_{l'}$. Then the unique ascending chain for $M/K$ is $K=F_0 F'_0\subsetneq F_1 F'_1 \subsetneq F_2 F'_2 \subsetneq \dots \subsetneq F_{l'} F'_{l'}\subsetneq F_{l'+1} F'_{l'}\subsetneq \dots  \subsetneq F_{l} F'_{l'}$ for $l\geq l'$ and $K=F_0 F'_0\subsetneq F_1 F'_1 \subsetneq F_2 F'_2 \subsetneq \dots \subsetneq F_{l} F'_{l}\subsetneq F_{l} F'_{l+1}\subsetneq \dots  \subsetneq F_{l} F'_{l'}$ for $l< l'$.
         
     \end{enumerate}
     
    \end{proposition}

\begin{proof} \hfill
    \begin{enumerate}
        \item By Proposition \ref{r prop} (1), $N_1=L^A$ and $N'_1=J^B$ where $A=Aut(L/K)$ and $B=Aut(J/K)$. Consider the unique intermediate extension $P_1/K$ of $M/K$ such that $M/P_1$ is Galois of maximum possible degree. So $P_1=M^C$ where $C=Aut(M/K)$. Let $x\in N_1$. Thus $x\in L$ such that $\lambda(x)=x$ for all $\lambda\in A$. Let $\sigma\in C$. By Lemma \ref{aut lem}, $\sigma_L\in A$. Thus $\sigma(x)=\sigma|_L(x)=x$. Thus $x\in P_1$. So $N_1\subset P_1$. Similarly we can show $N'_1\subset P_1$. So $N_1N'_1\subset P_1$.\smallskip
        
       Now suppose $y\in P_1$. Thus $y\in M=LJ$ such that $\sigma(y)=y$ for all $\sigma\in C$. Let $y=\Sigma_{i=1}^k\ l_i j_i$ and choose such a relation with minimal $k$. Thus both the sets $\{l_i\}_{i=1}^k$ and $\{j_i\}_{i=1}^k$ are linearly independent over $K$ (otherwise it will contradict minimality of $k$). Now $\Sigma\ \sigma(l_i)\sigma(j_i)=\Sigma\ l_i j_i$ for all $\sigma\in C$. By Lemma \ref{aut lem}, for any $\lambda \in A$ we have $\sigma\in C$ such that $\sigma_L=\lambda$ and $\sigma_J=id_J$. Thus for any $\lambda\in A$ we have $\Sigma\ \lambda(l_i)j_i=\Sigma\ l_i j_i$. Thus $\Sigma\ (\lambda(l_i)-l_i)j_i=0$. As $\tilde{L}$ and $\tilde{J}$ are linearly disjoint over $K$, $\{j_i\}_{i=1}^k$ are linearly independent over $\tilde{L}$. Hence we have for each $1\leq i\leq k$ that $\lambda(l_i)=l_i$ for all $\lambda\in A$. Thus $l_i\in N_1$ for all $i$. Similarly we can show $j_i\in N'_1$ for all $i$. Thus $y\in N_1 N'_1$. Thus $P_1=N_1N'_1$. Now the proof proceeds in the same way as proof of Theorem 3.8 (1) in \cite{jaiswal2025variantsinverseclustersize}.

        \smallskip

        \item The result is easy to see when $L_S=K$. For $[L_S:K]>1$, by Proposition \ref{GNM prop} (2), we have that $M_S/K$ is obtained by general normal magnification from $L_S/K$ through $J_S/K$. Thus the result follows from Proposition \ref{same uniq asc} and Theorem 3.8 (2) in \cite{jaiswal2025variantsinverseclustersize}.\end{enumerate}\end{proof}

\begin{remark}\label{ease notation}
    For ease of notation in statement of Theorem \ref{unique chain GNM}, for $k\geq k'$, one can define $N'_{k'+1}=N'_{k'+2}=\dots =N'_k=N'_{k'}$. Similarly one can define for $k<k',\ l\geq l',\ l<l'$. Thus the field in unique descending chain of $M/K$ at $i$-th step (where $i\leq max\{k,k'\}$) is simply $N_iN'_i$. Similarly we have for unique ascending chain.  
\end{remark}

The following generalizes Proposition 7.1.4 in \cite{Bhagwat_2025}.

\begin{proposition}
\label{symmetric sum}

Let $L/K$ be a nontrivial finite extension and $N$ be the unique intermediate extension for $L/K$ such that $L/N$ is Galois of maximum possible degree. Suppose $L=K(\alpha_1,\dots, \alpha_k)$ with $\alpha_i\in \bar{K}$. Let $f_i$ be respective minimal polynomial of $\alpha_i$ over K. Let $\rho_K(L,K(\alpha_i))=j_i$ and $\alpha_{i1},\dots , \alpha_{ij_i}$ be the distinct roots of $f_i$ contained in $L$. Then $$N=K(t_{11},\dots, t_{1j_1}, t_{21},\dots, t_{2j_2},\dots, t_{k1},\dots,t_{kj_k}).$$ where $t_{i1},\dots, t_{ij_i}$ are elementary symmetric sums of $\alpha_{i1},\dots, \alpha_{ij_i}$    
\end{proposition}

\begin{proof}
    Now $N=L^A$ where $A=Aut(L/K)$. Let $M=K(t_{11},\dots, t_{1j_1}, t_{21},\dots, t_{2j_2},\dots, t_{k1},\dots,t_{kj_k})$. Each $\alpha_i$ satisfies the polynomial $g_i(x)=(x-\alpha_{i1})\cdots (x-\alpha_{ij_i})=x^{j_i}-t_{i1} x^{j_i-1} + \dots + (-1)^{j_i} t_{1j_i}$ over $M$. Clearly $M\subset N$. Also $L/M$ is both normal and separable, hence Galois, so we are done.
\end{proof}

\begin{corollary}
    Consider a simple extension $L/K$ with $L=K(\alpha)$ for $\alpha\in \bar{K}$, we have $(\tilde{L})_I=K(t_1,\dots , t_k)$ where $k=\rho_K(\tilde{L},L)=[L_S:K]$ and $t_i$ are elementary symmetric functions of the $k$ distinct roots of minimal polynomial $f$ of $\alpha$ over $K$.
\end{corollary}

\begin{proof}
    By Corollary \ref{LI cor}, $\tilde{L}^{Aut(\tilde{L}/K)}=(\tilde{L})_I$. Let $\alpha_1,\dots\alpha_k$ be the distinct roots of $f$. Then $\tilde{L}=K(\alpha_1,\dots, \alpha_k)$. By Proposition \ref{symmetric sum}, we are done.
\end{proof}

\begin{example}

Consider an imperfect field $K$ with $char(K)=p>0$ and let $L=K(\gamma)/K$ with $L_S=K(\gamma^{p^u})$ and $[L:L_S]=p^u$ and $[L_S:K]=n>2$. Let $g$ be minimal polynomial of $\gamma$ over $K$ and let $\gamma_1,\gamma_2, \dots, \gamma_n$ be distinct roots of $g$ in $\bar{K}$ where $\gamma=\gamma_1$. Now $g=h(x^{p^u})$ where $h$ is a separable polynomial over $K$ with distinct roots $\gamma_1^{p^u},\dots, \gamma_n^{p^u}$.\smallskip

For $1\leq k\leq n$, let $L_k=K(\gamma_1,\gamma_2,\dots, \gamma_k)$ and let $M_k=K(\gamma_1^{p^u},\gamma_2^{p^u},\dots, \gamma_k^{p^u})$. Now $L_k/M_k$ is purely inseparable and thus $(L_k)_S=M_k$. Also as $L_1/M_1$ is purely inseparable and $M_k/M_1$ is separable, we have $[L_k:M_k]\geq [L_1M_k:M_k]=[L_1:M_1]=p^u$. Let $t_i$ be elementary symmetric sums of $\gamma_i$ for $1\leq i\leq k$. Thus $t_i^{p^u}$ are elementary symmetric sums of $\gamma_i^{p^u}$ for $1\leq i\leq k$. Let $N_k=K(t_1,t_2,\dots, t_k)$. Thus $(N_k)_S=K(t_1^{p^u},t_2^{p^u},\dots, t_k^{p^u})$. We also have $L_k$ to be the splitting field of the separable polynomial $x^{k}-t_{1} x^{k-1} + \dots + (-1)^{k} t_{k}$ over $N_k$ and thus $L_k/N_k$ is Galois and $r_K(L_k)\geq [L_k:N_k]$. Similarly $M_k/(N_k)_S$ is Galois and $r_K(M_k)\geq [M_k:(N_k)_S]$. As $L_k/(M_kN_k)$ is both separable and purely inseparable, we have $L_k=M_kN_k$. As $M_k/(N_k)_S$ is separable and $N_k/(N_k)_S$ is purely inseparable, we have $[L_k:N_k]=[M_k:(N_k)_S]$.\smallskip

 Now suppose $\tilde{L_S}/K$ has Galois group $\mathfrak{S}_n$. Then clearly $r_K(L_S)=1$. Thus $r_K(L)=1$ and $K(\gamma_i)\neq K(\gamma_j)$ for $i\neq j$. For $k\leq n-2$, we have that $(N_k)_S=M_k^{Aut((M_k/K)}$ by Theorem 7.3.6 (1)(i) in \cite{Bhagwat_2025}. Also $[M_k:(N_k)_S]=k!$ and $[(N_k)_S:K]= {n\choose k}$. So $r_K(M_k)=k!$. Also $r_K(L_k)\geq [L_k:N_k]=[M_k:(N_k)_S]=r_K(M_k)$. Thus $r_K(L_k)=k!=r_K((L_k)_S)$ and $L_k^{Aut(L_k/K)}=N_k$. Also for $1\leq k\leq n-2$, $L_k\neq L_{k+1}$ as $(L_k)_S=M_k\neq M_{k+1}=(L_{k+1})_S$.

\end{example}

\subsection{Unique Normal Chains and General normal magnification}

Now we define the notions of unique normal descending chains and unique normal ascending chains for extensions. Let $L/K$ be a nontrivial finite extension. \smallskip

In light of Proposition \ref{unique desc normal}. we have a unique normal strictly descending chain of subextensions \[ L=N'_0\supsetneq N'_1 \supsetneq N'_2 \supsetneq \dots \supsetneq N'_{k'}\]
        such that for all $1\leq i\leq k'$, $N'_i$ is the unique intermediate extension for $N'_{i-1}/K$ such that $N'_{i-1}/N'_i$ is normal of maximum possible degree. 
\smallskip

Similarly in light of Proposition \ref{ad prop}, we have a unique normal strictly ascending chain of subextensions  \[ K=F'_0\subsetneq F'_1 \subsetneq F'_2 \subsetneq \dots \subsetneq F'_{l'} \] such that for all $1\leq j\leq l'$, $F'_j$ is the unique intermediate extension for $L/F'_{j-1}$ such that $F'_j/F'_{j-1}$ is normal of maximum possible degree.\smallskip

Both the unique chains terminate since $L/K$ is finite. Also $d_K(N'_{k'})=1$ and $a_{F'_{l'}}(L)=1$.

\begin{proposition}\label{uniq norm chain prop} Consider $L/K$.

\begin{enumerate}

\item If $d_K(L)>1$, then the unique normal descending chain for $L/K$ is $$L\supsetneq  \text{unique descending chain for} (L^A)_S.$$

\item If the unique normal ascending chain for $L/K$ is \[ K=F'_0\subsetneq F'_1 \subsetneq F'_2 \subsetneq \dots \subsetneq F'_{l'} \]

and the unique ascending chain for $L/K$ is \[ K=F_0\subsetneq F_1 \subsetneq F_2 \subsetneq \dots \subsetneq F_l \]

then $l\leq l'$ and for $1\leq i\leq l'$ we have $F'_i=F_i (L_I)_i$ where $(L_I)_i/F'_{i-1}$ is the maximal purely inseparable subextension of $L/F'_{i-1}$ and $F_i=F_l$ for $i\geq l$.\smallskip

 Futhermore if $L=L_SL_I$ then we have $l=l'$ and $F'_i=F_i L_I$.

\end{enumerate}
\end{proposition}

\begin{proof}
    We prove part (2). Now $F'_1=F_1(L_I)_1$ by  By Proposition \ref{ad prop} (1). Now $L_S(L_I)_1/F'_1$ is the maximal separable subextension of $L/F'_1$. Thus the unique ascending chain for $L/F'_1$ is \[ F'_1=F_1 (L_I)_1\subsetneq F_2 (L_I)_1\subsetneq F_3 (L_I)_1\subsetneq \dots \subsetneq F_l (L_I)_1.\] Thus by Proposition \ref{ad prop} (1), $F'_2= (F_2(L_I)_1)(L_I)_2=F_2(L_I)_2$ as $(L_I)_2\supset F'_1=F_1(L_I)_1\supset (L_I)_1$.
Proceeding similalry we are done.\smallskip

If $L=L_SL_I$, then $L/L_I$ is separable thus $L/F_{i-1}$ is separable for all $i\geq 2$. Thus for $i\geq 2$, $(L_I)_i=F'_{i-1}$ and $F'_i=F_i F'_{i-1}$. As $F'_1=F_1 L_I$, so $F'_i=F_iL_I$.\end{proof}

\begin{proposition}

    The unique normal ascending chain of $L/K$ is the unique normal ascending chain of $L^A/K$ $\iff$ the unique ascending chain of $L/K$ is the unique ascending chain of $L^A/K$.

\end{proposition}

\begin{proof}

    First suppose $F'_i\subset L^A$ for all $i$. Then clearly $F_i\subset L^A$ for all $i$. Conversely, suppose $F_i\subset L^A$ for all $i$. Since $F_1\in L^A$ and $L_I\subset L^A$. Thus $F'_1=F_1 L_I\subset L^A$. Thus $L^{Aut(L/F'_1)}=L^A$. So $(L_I)_2\subset L^A$. By Proposition \ref{uniq norm chain prop} (2), $F'_2=F_2 (L_I)_2$. As $F_2\subset L^A$, so $F'_2\subset L^A$. Proceeding similarly we are done.\end{proof}

\begin{corollary}\label{normal chains}
    
 Suppose $L/K$ is nontrivial normal extension. Then the unique normal descending chain is $L\supsetneq K$ and the unique normal ascending chain is $K\subsetneq L$. If $L/K$ is purely inseparable, then unique descending chain is singleton $L$ and unique ascending chain is singleton $K$. If $L/K$ is not purely inseparable, then unique descending chain is $L\supsetneq L^A=L_I$ and unique ascending chain is $K\subsetneq L_S$.

\end{corollary}

\begin{proposition}
    
\label{unique normal chain GNM}

   Suppose the extension $M/K$ is obtained by general normal magnification from  $L/K$ through $J/K$. 
     
     \begin{enumerate}
\item Let the unique normal descending chain for $L/K$ be $L=N_0\supsetneq N_1 \supsetneq N_2 \supsetneq \dots \supsetneq N_k$ and the unique normal descending chain for $J/K$ be $J=N'_0\supsetneq N'_1 \supsetneq N'_2 \supsetneq \dots \supsetneq N'_{k'}$. Then the unique normal descending chain for $M/K$ is $M=N_0 N'_0\supsetneq N_1 N'_1 \supsetneq N_2 N'_2 \supsetneq \dots \supsetneq N_{k'} N'_{k'}\supsetneq N_{k'+1} N'_{k'}\supsetneq \dots  \supsetneq N_{k} N'_{k'}$ for $k\geq k'$ and $M=N_0 N'_0\supsetneq N_1 N'_1 \supsetneq N_2 N'_2 \supsetneq \dots \supsetneq N_{k} N'_{k}\supsetneq N_{k} N'_{k+1}\supsetneq \dots  \supsetneq N_{k} N'_{k'}$ for $k<k'$.

\smallskip

\item Let the unique normal ascending chain for $L/K$ be $K=F_0\subsetneq F_1 \subsetneq F_2 \subsetneq \dots \subsetneq F_l$ and the unique normal ascending chain for $J/K$ be $K=F'_0\subsetneq F'_1 \subsetneq F'_2 \subsetneq \dots \subsetneq F'_{l'}$. Then the unique normal ascending chain for $M/K$ is $K=F_0 F'_0\subsetneq F_1 F'_1 \subsetneq F_2 F'_2 \subsetneq \dots \subsetneq F_{l'} F'_{l'}\subsetneq F_{l'+1} F'_{l'}\subsetneq \dots  \subsetneq F_{l} F'_{l'}$ for $l\geq l'$ and $K=F_0 F'_0\subsetneq F_1 F'_1 \subsetneq F_2 F'_2 \subsetneq \dots \subsetneq F_{l} F'_{l}\subsetneq F_{l} F'_{l+1}\subsetneq \dots  \subsetneq F_{l} F'_{l'}$ for $l< l'$.
         
     \end{enumerate}
     
    \end{proposition}

\begin{proof}
    \hfill

\begin{enumerate}
\item By Proposition \ref{unique desc normal} (3), $N_1=(L^A)_S$ and $N'_1=(J^B)_S$ where $A=Aut(L/K)$ and $B=Aut(J/K)$. Consider the unique intermediate extension $P_1/K$ of $M/K$ such that $M/P_1$ is normal of maximum possible degree. So $P_1=(M^C)_S$ where $C=Aut(M/K)$. By Proposition \ref{unique chain GNM} (1), $M^C=(L^A)(J^B)$. Thus by Lemma \ref{compositum} (1), $P_1=N_1N'_1$. Thus $P_1/K$ is obtained by general normal magnification from $N_1/K$ through $N'_1/K$. Since $P_1/K,\ N_1/K,\ N'_1/K$ are separable, their unique normal descending chains are precisely their unique descending chains and thus we are done by Proposition \ref{unique chain GNM} (1).

\smallskip

\item

By Proposition \ref{ad prop} (1), $F_1=F L_I$, $F'_1=F' J_I$ and $P_1=P M_I$ where $P_1/K$ is maximal normal subextension of $M/K$ and $F/K,\ F'/K, P/K$ are respective maximal Galois subextensions of $L/K,\ J/K,\ M/K$. By Proposition \ref{unique chain GNM} (2), $P=FF'$. By Proposition \ref{GNM prop} (4), $M_I=L_IJ_I$. Thus $P_1=F_1 F'_1$. \smallskip

Now as $\tilde{L}$ and $\tilde{J}$ are linearly disjoint over $K$, so $\tilde{L}$ and $P_1$ are linearly disjoint over $F_1$. Also $P_1/F_1$ is normal. Thus $LP_1/F_1$ is obtained by general normal magnification from $L/F_1$ through $P_1/F_1$. By Corollary \ref{isom lem cor}, all the distinct fields isomorphic to $LP_1/P_1$ are $L_iP_1$ where $L_i/F_1$ are all the distinct fields isomorphic to $L/F_1$. Thus by Proposition \ref{f'f} (1), the maximal normal subextension of $LP_1/P_1$ is $\cap_i\ (L_iP_1)$. Also $F_2=\cap_i\ L_i$. By Lemma \ref{lin disj lem}, $\cap_i\ (L_iP_1)=(\cap_i\ L_i) P_1=F_2P_1$. Similarly the maximal normal subextension of $JP_1/P_1$ is $F'_2P_1/P_1$. Let $P_2/P_1$ be the maximal normal subextension of $M/P_1$. By Proposition \ref{hereditary}, $M/P_1$ is obtained by general normal magnification from $LP_1/P_1$ through $JP_1/P_1$. Thus by previous paragraph, $P_2=(F_2P_1)(F'_2P_1)=F_2F'_2$. Proceeding similarly we are done.
\end{enumerate}
\end{proof}

The following is an interesting generalization of Theorem 2.4 in \cite{jaiswal2025variantsinverseclustersize}.

\begin{proposition}
     Consider a nontrivial finite extension $M/K$ that is obtained by general normal magnification from subextension $L/K$ through a normal subextension $J/K$.
     
   \begin{enumerate}
\item Suppose one of the fields in the unique descending chain of $M/K$ which is different from both $M$ and $K$, coincides with one of the fields in the unique ascending chain of $M/K$. Then $J/K$ is trivial and $M/K$ is separable.\smallskip

\item Suppose one of the fields in the unique normal descending chain of $M/K$ which is different from both $M$ and $K$, coincides with one of the fields in the unique normal ascending chain of $M/K$. Then $J/K$ is trivial and $i_K(M)=1$.\smallskip

\item  Suppose one of the fields in the unique descending chain of $M/K$ which is different from both $M$ and $K$, coincides with one of the fields in the unique normal descending chain of $M/K$. Then $J/K$ is Galois and $M/K$ is separable.

\smallskip

\item Suppose one of the fields in the unique ascending chain of $M/K$ which is different from both $M$ and $K$, coincides with one of the fields in the unique normal ascending chain of $M/K$. Then $J/K$ is Galois and $i_K(M)=1$.

\smallskip

\item Suppose either 

\begin{enumerate}
\item one of the fields in the unique descending chain of $M/K$ which is different from both $M$ and $K$, coincides with one of the fields in the unique normal ascending chain of $M/K$, or

\item one of the fields in the unique normal descending chain of $M/K$ which is different from both $M$ and $K$, coincides with one of the fields in the unique ascending chain of $M/K$.

\end{enumerate}

Then $J/K$ is purely inseparable.

\end{enumerate}
\end{proposition}

\begin{proof}
    Let the unique descending chain for $L/K$ be $L=N_0\supsetneq N_1 \supsetneq N_2 \supsetneq \dots \supsetneq N_k$ and the unique ascending chain for $L/K$ be $K=F_0\subsetneq F_1 \subsetneq F_2 \subsetneq \dots \subsetneq F_l$ and the unique normal descending chain for $L/K$ be $L=N'_0\supsetneq N'_1 \supsetneq N'_2 \supsetneq \dots \supsetneq N'_{k'}$ and the unique normal ascending chain for $L/K$ be $K=F'_0\subsetneq F'_1 \subsetneq F'_2 \subsetneq \dots \subsetneq F'_{l'}$. By Proposition \ref{unique chain GNM} and Proposition \ref{unique normal chain GNM} and Corollary \ref{normal chains}, the unique descending chain for $M/K$ is $M=LJ\supsetneq N_1J_I \supsetneq N_2J_I \supsetneq \dots \supsetneq N_kJ_I$ and the unique ascending chain for $M/K$ is $K\subsetneq F_1J_S \subsetneq F_2 J_S\subsetneq \dots \subsetneq F_lJ_S$ and the unique normal descending chain for $M/K$ is $M=LJ\supsetneq N'_1 \supsetneq N'_2 \supsetneq \dots \supsetneq N'_{k'}$ and the unique normal ascending chain for $M/K$ is $K\subsetneq F'_1J \subsetneq F'_2J \subsetneq \dots \subsetneq F'_{l'}J$. Also for any $K\subset L'\subset L$ and $K\subset J'\subset J$ we have $\tilde{L}\cap L'J'=L'$ and $\tilde{J}\cap L'J'=J'$.

\begin{enumerate}
\item If $N_iJ_I=F_jJ_S$, then $J_I=J_S$. Thus $J=K$ and $M=LJ=L$. Also $N_i=F_j\subset L_S$. By Proposition \ref{LSNi} (1), $L=L_SN_i=L_S$.

\smallskip

\item  If $N'_i=F'_jJ$, then $J=K$ and $M=L$. Also $F'_j=N'_i\subset L_S$. So $F'_1=F_1L_I\subset L_S$ and hence $L_I=K$.

\smallskip

\item  If $N_iJ_I=N'_j$, then $J_I=K$. So $J=J_SJ_I=J_S$. Also $N_i=N'_j\subset L_S$. Thus $L=L_SN_i=L_S$. So $M=LJ=L_SJ_S=(LJ)_S=M_S$.

\smallskip

\item  If $F_iJ_S=F'_j J$, then $J=J_S$. So $J_I=K$. Also $F'_j=F_i\subset L_S$. So $F'_1=F_1L_I\subset L_S$ and hence $L_I=K$. So $M_I=(LJ)_I=L_IJ_I=K$.

\smallskip

\item (a) If $N_i J_I= F'_j J$, then $J=J_I$, (b) If $N'_i=F_jJ_S$, then $J_S=K$. So $J=J_SJ_I=J_I$.

\end{enumerate}
    
\end{proof}

\section{Inverse Problems over Hilbertian fields}\label{inv prob section}

We begin with some lemmas.

\begin{lemma}\label{perm}
(Reformulation of Final Proposition \cite{perlisroots}) Let $G$ be a transitive subgroup of ${\mathfrak S}_n$ for some $n$. If there exists a finite Galois extension of a field $K$ with Galois group isomorphic to $G$, then there exists an extension $L/K$ of degree $n$ with Galois closure having Galois group $G$ over $K$.

\end{lemma}

Let $K$ be a field. An element $\sigma\in GL_n(K)$ is called a pseudo-reflection if $ord(\sigma)$ is finite and $dim(ker(\sigma-1))=n-1$ (See \S 7.1, \cite{benson1993polynomial}). A group $G$ is said to be regular over $K$ if there exist indeterminates $t_1,\dots, t_r$ over $K$ such that $K(t_1,\dots,t_r)$ has a Galois extension $F$ which is regular over $K$ (i.e. $F/K$ is separable and $K$ is algebraically closed in $F$) with $Gal(F/K(t_1,\dots, t_r))\cong G$ (See \S 16.2, \cite{fried2005field}). We say that $K$ is Hilbertian if each separable Hilbert set of $K$ is non-empty (Refer \S 12.1, \cite{fried2005field} for a detailed definition). By results in Section 13.2 in \cite{fried2005field}, function fields of several variables over infinite fields are Hilbertian. By results in Section 13.3 in \cite{fried2005field}, global fields are Hilbertian.

\begin{lemma}\label{hilb fie}
    Let $K$ be a Hilbertian field. We have finite group $G$ to be regular over $K$ and hence realizable as a Galois group for infinitely many pairwise linearly disjoint Galois extensions over $K$ for the following cases.

    \begin{enumerate}
\item  $G={\mathfrak S}_{n}$ for any $n\in \mathbb{N}$.

\item  $G={\mathfrak A}_{n}$ for any $n\in \mathbb{N}$ when $char(K)\neq 2$ and for $n$ odd when $char(K)= 2$.

\item  $G$ which is any finite abelian group of order $n$ for any $n\in \mathbb{N}$.

\item Any finite group $G\subseteq GL_n(K)$ generated by pseudo-reflections with $|G|$ invertible in $K$.

\item $G$ which is any finite $p$-group when $char(K)>0$.

\item $G=A\rtimes B$ where $A$ is any finite abelian group and $B$ is regular over $K$ and $B$ acts on $A$. 


\end{enumerate}
\end{lemma}

\begin{proof}

\hfill

\begin{enumerate}
\item  Corollary 16.2.7 (a) in \cite{fried2005field}.  \smallskip


\item By the results in \cite{brink2004alternating} on realizability of $\mathfrak{A}_n$ as a Galois group over Hilbertian fields (when $n\in\mathbb{N}$ and $char(K)\neq 2$ or when $n$ is odd and $char(K)=2$) and by results in \S 16.2 in \cite{fried2005field} on Hilbertian fields, we have for $G={\mathfrak A}_{n}$.\smallskip

\item Corollary 16.3.6 in \cite{fried2005field}.\smallskip

\item Consider the ring $K[X_1,X_2,\dots, X_n]$ (where $X_i$'s are indeterminates) and let $G$ act linearly on the ring (fixing $K$). By the Chevalley-Shephard-Todd Theorem (Theorem 7.2.1 in \cite{benson1993polynomial}) we have that $(K[X_1,X_2,\dots, X_n])^G$ is a polynomial ring over $K$ i.e. $(K[X_1,X_2,\dots, X_n])^G=K[Y_1,Y_1,\dots, Y_n]$ (where $Y_i$'s are algebraically independent). Thus by Proposition 1.2.4 in \cite{smith1995polynomial}, we have that $K(X_1,X_2,\dots, X_n)/K(Y_1,Y_1,\dots, Y_n)$ is a finite Galois extension with Galois group $G$.  As $K$ is Hilbertian, by results in \S 16.2 in \cite{fried2005field}, we are done.
\smallskip

\item Theorem 16.4.7 in \cite{fried2005field}.\smallskip

\item By Corollary 16.4.8 in \cite{fried2005field} and by results in \S 16.2 in \cite{fried2005field}.

\end{enumerate}\end{proof}

By applying Lemma \ref{perm}, we have the following.

\begin{corollary}\label{Sn ref}
 Let $K$ be a Hilbertian field. Then there exist infinitely many pairwise linearly disjoint extensions over $K$ of degree $n$ with Galois closures having Galois group $G$ over $K$ for the cases in part (1), (2) and (3) above.

\end{corollary}

\begin{lemma}\label{lin disj}
    Given separable $L/K$. Suppose $G$ is regular over $K$. If $K$ is Hilbertian then we have an extension $L'/K$ such that $Gal(L'/K)=G$ and $L'$ and $L$ are linearly disjoint over $K$.
   
\end{lemma}

\begin{proof}
 Since $G$ is regular over $K$ which is Hilbertian, there are infinitely many pairwise linearly disjoint extension $L'/K$ such that $Gal(L'/K)=G$. Call the infinite set of these extensions as $\mathfrak L$. Since $L/K$ is separable, it has finitely many intermediate extensions. Consider the subset $\mathfrak L'$ of $\mathfrak L$ consisiting of those $L'/K$ which contain atleast one non-trivial intermediate extension of $L/K$. Due to pairwise linear disjointness, $\mathfrak L'$ is finite. Thus we have an extension $L'/K$ such that $Gal(L'/K)=G$ and $L'\cap L=K$. As $L'/K$ is Galois so $L'$ and $L$ are linearly disjoint over $K$.
\end{proof}

\begin{lemma}\label{p insep}
    Consider an imperfect field $K$ with $char(K)=p$. Then 
    
\begin{enumerate}
\item  For any $u\in \mathbb{N}\cup \{0\}$, there exists a purely inseparable simple extension $L/K$ of degree $p^u$.

\smallskip

 \item   If $M/K$ is a finite extension, then $M$ is imperfect.

\end{enumerate}
\end{lemma}

\begin{proof}

Since $K$ is imperfect, there is an $a\in K$ such that $a\not \in K^p$. 

\begin{enumerate}
    \item 

 Consider the polynomial $f(x)=x^{p^u}-a$ over $K$. Let $L=K(\alpha)$ where $\alpha\in \bar{K}$ is a root of $f$. We claim that $f$ is irreducible over $K$. Now $f(x)=x^{p^u}-a=(x-\alpha)^{p^u}$. If $f$ is reducible, say $f(x)=g(x)h(x)$, then $g(x)=(x-\alpha)^m\in K[x]$ for some $1\leq m<p^u$. Thus $\alpha^m\in K$. As $a=\alpha^{p^u}\in K$, we have $\alpha^{p^v}\in K$ where $p^v=gcd(p^u,m)$. As $v<u$, we get $a\in K^p$ which is a contradiction. Thus $L/K$ is the required extension.\smallskip

\item Suppose $M$ is perfect. As $a\in M$ and $M=M^p$, so $a^{1/p}\in M$. Similarly we have $a^{1/p^n}\in M$ for all $n\in \mathbb{N}$. By proof of part (1), $[K(a^{1/p^n}):K]=p^n$ for all $n$. Thus $p^n\mid [M:K]$ for all $n$ which is a contradiction.

\end{enumerate}
\end{proof}

\begin{remark}
  Global function fields are finite extensions of $\mathbb{F}_q(t)$ where $q=p^k$ for $p$ prime and $k\geq 1$ and $t$ is an indeterminate. Since $t^{1/p}\not \in \mathbb{F}_q(t)$, so $F_q(t)$ is imperfect. By Lemma \ref{p insep} (2), global function fields are imperfect.
\end{remark}

\begin{lemma}

Consider the field $K=\mathbb{F}_q (t)$ where $q=p^m$ and $m$ is a positive integer and $t$ is an indeterminate. Then for any $n\geq 1$, the group $G_n=GL_n(\mathbb{F}_q)$ is realizable as a Galois group over $K$ for infinitely many pairwise linearly disjoint Galois extensions over $K$.
\end{lemma}

\begin{proof}

For $n=1$, $G_1=\mathbb{F}_q^{\times}$ is cyclic. As $K$ is Hilbertian, we have $G_1$ to be realizable as a Galois group over $K$ by Lemma \ref{hilb fie} (3). Now let $n\geq 2$. By Proposition 8.1.4 in \cite{smith1995polynomial}, $G_n$ is realizable as a Galois group over $K(t_1,t_2,\dots, t_{n-1})$ where $t_i$'s are algebraically independent over $K$. As $K$ is Hilbertian, by results in \S 16.2 in \cite{fried2005field}, we have the result.\end{proof}

\begin{corollary}
    For any $n\geq 1$, $PGL_n(\mathbb{F}_q)$ is realizable as a Galois group over $K$ for infinitely many pairwise linearly disjoint Galois extensions over $K=\mathbb{F}_q(t)$.
\end{corollary}

\begin{remark}
    For any extensions $L/K$ and $M/K$, if $\lambda_K(M,L)=[L:K]-[L:L_S]$ then $r_K(L)=1$. Dividing by $[L:L_S]$ we have $\rho_K(M,L)=[L_S:K]-1$. As $\rho_K(M,L)=a\cdot r_K(L)$ for some $a\leq s_K(L)$ and $[L_S:K]=r_K(L)\cdot s_K(L)$, so we have $(s_K(L)-a)\cdot r_K(L)=1$. Thus $r_K(L)=1$ and $s_K(L)=[L_S:K]$ and $a=[L_S:K]-1$.
\end{remark}

\begin{proof}[Proof of Theorem \ref{root cap gen res}]
We prove the theorem in two parts.

\begin{enumerate}
\item First we will show that: Given $(n,r,\rho)$ where $n>2$ and $r\mid n$ and $r\mid \rho$ and $\rho\leq n$ and $\rho\neq n-1$. There exist separable extensions $L/K$ and $M/K$ such that $[L:K]=n$ and cluster size $r_K(L)=r$ and root capacity $\rho_K (M,L)=\rho$. For $\rho\neq 0$, we get $M/K$ as an extension of $L/K$ contained in $\tilde{L}$.\smallskip

Let $n/r=s$ and $\rho/r = a$. For $r=1$, we have $\rho=a$ and so $a\neq n-1$. By Corollary \ref{Sn ref}, there exists an $L/K$ of degree $n$ with Galois closure having Galois group ${\mathfrak S}_n$. This satisfies $r_K(L)=1$ and $s_K(L)=n$. Let $L_1,L_2,\dots, L_n$ be distinct fields $K$-isomorphic to $L$. Then for $a=0$, $M=K$ works. Let $M=L_1\cdots L_a$ for $1\leq a\leq n-2$ and let $M=L_1\cdots L_{n-1}$ for $a=n$. Thus $\rho_K(M,L)=a$.
 
 \smallskip

 Now for $r>1$. We construct a solvable group $G\subseteq{\mathfrak S}_n$ (as in solutions of Exercises 3 and 4 in \cite{perlisroots}) with the following properties: its action is transitive on $n$ points, and a point stabiliser fixes precisely $r$ points. We divide the $n$ points into $n/r=s$ packets of size $r$. Let $G$ be the group of permutations on these $n$ points generated by independent
cyclic permutations on each packet, together with a cyclic permutation
on the overall set of packets. Thus, we have $G=(\Z/r\Z)^s \rtimes \Z / s\Z$ with semidirect product group law \[ ((a_1,\dots, a_s),b)\cdot((c_1,\dots,c_s),d)=((a_1,\dots, a_s) + (b\cdot (c_1,\dots, c_s)), b+d), \]
  where $b\cdot(c_1,\dots, c_s)=(c_{b+1},\dots, c_s,c_1,\dots, c_{b})$ for $b\neq 0$ \&  $0\cdot(c_1,\dots, c_s)=(c_1,\dots, c_s)$.\smallskip
  
  One can verify $((a_1,\dots,a_s),b)^{-1}=((-a_{s-b+1},\dots,-a_{s}, -a_{1},\dots, -a_{s-b}),-b)$ and \[((a_1,\dots, a_s),b)\cdot((c_1,\dots,c_{s-1},0),0)\cdot ((a_1,\dots,a_s),b)^{-1}= ((c_{b+1},\dots, c_{s-1}, 0, c_1,\dots,c_b),0).\]

 Any point stabiliser is isomorphic to $(\Z/r\Z)^{s-1}$. Now $(\Z/r\Z)^s$ is finite abelian. Also $\Z / s\Z$ is regular over $K$ by Lemma 16.3.4 in \cite{fried2005field} and it clearly acts on $(\Z/r\Z)^s$ as above. Thus by Lemma \ref{hilb fie} (6), $G$ is realizable as a Galois group over $K$. By Lemma \ref{perm}, there exists a separable extension $L/K$ with $[L:K]=n$ and with Galois closure having Galois group $G$ over $K$. This $L/K$ satisfies $r_K(L)=r$. The $s=s_K(L)$ many subgroups of $G$ fixing the $s$ many distinct fields $L_i$'s isomorphic to $L/K$ are $H_i=((\Z/r\Z)^{i-1}\times 0 \times (\Z/r\Z)^{s-i} )\times 0$ for $1\leq i\leq s$. Observe that $G \supsetneq H_1 \supsetneq H_1 \cap H_2 \supsetneq \dots \supsetneq H_1 \cap H_2 \cap \cdots \cap H_s = 0$. Let $M=L_1L_2\cdots L_a$ for $1\leq a \leq s$. Thus $\rho_K(M,L)=a\cdot r_K(L)=a\cdot r=\rho$.

 \medskip

 \item Now we prove the theorem. We break the problem into cases.\smallskip 
 
 Case 1: $n=p^{\mu}$. Now $l=r\cdot p^{\mu}\mid n$. So $r=1$ and $l=p^{\mu}=n$. As $l\mid \lambda$ and $\lambda\leq n$. So $\lambda=l$. Thus $\rho=r$. As $K$ is imperfect, by Lemma \ref{p insep} (1), we have a purely inseparable simple extension $L/K$ where $L=K(\alpha)$ for some $\alpha\in \bar{K}$ with $[L:K]=p^{\mu}=n$. Thus $r_K(L)=1$ and $l_K(L)=l$. As $n>1$, so for $M=K$ we have $\rho_K(M,L)=0=\lambda_K(M,L)$. For $M=L$ we have $\rho_K(M,L)=r_K(L)=r$ and $\lambda_K(M,L)=l_K(L)=l$.

 \smallskip
 
Case 2: $n>2p^{\mu}$.  Let $n'=n/p^{\mu}$. By the assumptions, we have $n'>2$ and $r\mid n'$ and $r\mid \rho$ and $\rho\leq n'$ and $\rho\neq n'-1$. As $K$ is Hilbertian, so by part (1), we have separable extensions $S/K$ and $M'/K$ with $S\subset M'\subset \tilde{S}$ such that $[S:K]=n'$ and $r_K(S)=r$ and $\rho_K(M',S)=\rho$. As $K$ is imperfect, by Lemma \ref{p insep} (1), we have a purely inseparable simple extension $I/K$ with $[I:K]=p^{\mu}$.  Let $L=SI$. So $L/S$ is purely inseparable and $L/I$ is separable. Also $S\subset L_S$ and $I\subset L_I$. As $L_S/S$ and $L_I/I$ are both separable and purely inseparable, so $L_S=S$ and $L_I=I$ and $L=L_SL_I$. By Proposition \ref{L'S prop} (3), $[L:L_S]=i_K(L)=[L_I:K]=p^{\mu}$. By Proposition \ref{L'S prop} (4), $r_K(L)=r_K(L_S)$ and thus $s_K(L)=s_K(L_S)$. So $r_K(L)=r$ and $l_K(L)=r_K(L)\cdot [L:L_S]=r\cdot p^{\mu}=l$. By Proposition \ref{gen norm for L} (1) and Proposition \ref{GNM thm II},  $\rho_K(M,L)=\rho_K(M',S)\cdot \rho_K(I,I)=\rho\cdot 1=\rho$ and $\lambda_K(M,L)=\lambda_K(M',S)\cdot \lambda_K(I,I)=\rho\cdot p^{\mu}=\lambda$.

\smallskip

Case 3: $n=2p^{\mu}$. Let $a=\rho/r=\lambda/l$. As $\lambda\leq n$, so $l\cdot a\leq 2p^{\mu}$. Thus $r\cdot p^{\mu}\cdot a\leq 2p^{\mu}$ and so $r\cdot a \leq 2$. As we have assumed $r\neq 1$ when $n=2p^{\mu}$, we have $r=2$. So $a=1$ and $\rho=2$ and $\lambda=l=2p^{\mu}=n$. By Lemma \ref{hilb fie} (3), $\mathbb{Z}/2\mathbb{Z}$ is realizable as a Galois group over $K$. Let $S/K$ be the degree $2$ Galois extension. So $r_K(S)=2$. As $K$ is imperfect, by Lemma \ref{p insep} (1), we have a purely inseparable simple extension $I/K$ with $[I:K]=p^{\mu}$.  Let $L=SI$. So $L_S=S$ and $L_I=I$ and $L=L_SL_I$ and $L/K$ is normal. Thus $[L:L_S]=i_K(L)=[L_I:K]=p^{\mu}$ and $[L:K]=2p^{\mu}$ and $r_K(L)=r_K(L_S)=2$ and thus $s_K(L)=s_K(L_S)=1$. So $l_K(L)=r_K(L)\cdot [L:L_S]=2\cdot p^{\mu}=n$. For $M=L$ we have $\rho_K(M,L)=r_K(L)=2$ and $\lambda_K(M,L)=l_K(L)=n$.

\smallskip

In Case 2 and Case 3, we observe that as $S/K$ is separable, so by primitive element theorem $S=K(\beta)$ for some $\beta\in \bar{K}$. As $I/K$ is a purely inseparable simple extension, $I=K(\gamma)$ for some $\gamma\in \bar{K}$. So $L=K(\beta, \gamma)$. By stronger version of primitive element theorem, Theorem 1 in \cite{brown2012mathematics} (or by Theorem C.1 in \cite{conradseparability}), $L=K(\alpha)$ for some $\alpha\in \bar{K}$.\end{enumerate}\end{proof}


\begin{proof}[Proof of Theorem \ref{tau gen res}]
    We prove the theorem in two parts.

\begin{enumerate}
\item First we will show that: Given $(n,t,\tau)$ where $n>2$ and $t\ |\ \tau\ |\ n$. Assume that in the case when $t=1$ we have $\tau\neq 2$ and $n\neq 2\tau$ and assume that in the case when $t$ is odd we have either $\tau\neq 2t$ or $n\neq 2\tau$. There exist separable extensions $L/K$ and $M/K$ such that $[L:K]=n$ and ascending index $t_K(L)=t$ and intersection indicium $\tau_K (M,L)=\tau$. We get $M/K$ as an extension of $L/K$ contained in $\tilde{L}$.

\smallskip

We break the problem into cases.\smallskip 

Case 1 : Suppose $\tau=t$ or $\tau=n$. We claim that given $n>2$ with $t\mid n$ we have an extension $L/K$ with $[L:K]=n$ and $t_K(L)=t$. For $t=1$, consider the $L/K$ for $r=1$ in proof of Theorem \ref{root cap gen res} (1). By Example 2.3 (2) in \cite{jaiswal2025rootcapacityintersectionindicium}, $t_K(L)=1=t$. For $t=n$, consider the $L/K$ for $r=n$ in proof of Theorem \ref{root cap gen res} (1). Since $L/K$ is Galois, $t_k(L)=n=t$. For $1<t<n$, consider the $L/K$ for $r=n/t$ in proof of Theorem \ref{root cap gen res} (1). By computation in Theorem 7.3.4 in \cite{Bhagwat_2025}, $r_K(L)\cdot t_K(L)=n$. Thus $t_K(L)=n/r=t$. Thus for $M=\tilde{L}$ we have $\tau_K(M,L)=t_K(L)=t$ and for $M=L$ we have $\tau_K(M,L)=[L:K]=n$. \smallskip

 Case 2 : Suppose $n/\tau>2$ and $\tau>2$. Then by Case 1, we have an extension $L'/K$ of degree $n/\tau$ such that $t_K(L')=1$. Since we have $t\mid \tau$ and $\tau>2$. Thus by Case 1, we have $J/K$ with $[J:K]=\tau$ and $t_K(J)=t$. Now by Lemma \ref{lin disj} we can choose the fields such that $\tilde{L'}\cap \tilde{J}=K$.
   Let $L=L'J$ and $M=\tilde{L'}J$. Thus by Proposition \ref{GNM theorem} (1), (3) and Proposition \ref{GNM thm II} (2), we have $[L:K]=[L':K][J:K]=(n/\tau)\cdot \tau=n$ and $t_K(L)=t_K(L')t_K(J)=1\cdot t=t$ and $\tau_K(M,L)=[J:K]t_K(L')=\tau\cdot1=\tau$.\smallskip

Case 3 : Suppose $n=2\tau$ and $t>1$ and $\tau/t>2$. Then we have an extension $J/K$ of degree $\tau/t$ such that $t_K(J)=1$. Since we have $t\mid 2t$ and $t>1$. Thus we have $L'/K$ with $[L':K]=2t$ and $t_K(L')=t$. Now we can choose the fields such that $\tilde{L'}\cap \tilde{J}=K$.
   Let $L=L'J$ and $M=\tilde{L'}J$. Thus we have $[L:K]=[L':K][J:K]=(2t)\cdot (\tau/t)=2\tau=n$ and $t_K(L)=t_K(L')t_K(J)=t\cdot 1=t$ and $\tau_K(M,L)=[J:K]t_K(L')=(\tau/t)\cdot t=\tau$.\smallskip

   Case 4 : Suppose $n=2\tau$ and $\tau=2t$ and $t=2m$ with $m>1$. So $n=8m$ and $\tau=4m$. We have extension $L'/K$ with $[L':K]=4$ and $t_K(L')=2$ and extension $J/K$ with $[J:K]=2m$ and $t_K(J)=m$. Now we can choose the fields such that $\tilde{L'}\cap \tilde{J}=K$.
   Let $L=L'J$ and $M=\tilde{L'}J$. Thus we have $[L:K]=[L':K][J:K]=4\cdot 2m=8m=n$ and $t_K(L)=t_K(L')t_K(J)=2\cdot m=2m=t$ and $\tau_K(M,L)=[J:K]t_K(L')=2m\cdot 2=4m=\tau$.
   
\smallskip

   Case 5 : $n=8,\tau=4,t=2$. First let $K=\mathbb{Q}$. Consider the polynomial $x^8-3$ which is irreducible over $\mathbb{Q}$ having a solvable Galois group by results in \cite{jacobson1990galois}. Let $a=3^{1/8}$ be the positive real root of the polynomial and $\iota$ be a primitive $4$-th root of unity. Let $L=\mathbb{Q}(a)$ and $M=\mathbb{Q}(a,\iota)$. Thus $\tau_{\mathbb{Q}}(M,L)=[\mathbb{Q}(a)\cap \mathbb{Q}(ai):\mathbb{Q}]=[\mathbb{Q}(a^2):\mathbb{Q}]=4$. Also by Example 2.3 (1) in \cite{jaiswal2025rootcapacityintersectionindicium}, $t_{\mathbb{Q}}(L)=[\mathbb{Q}(a^4):\mathbb{Q}]=2$.   By Proposition 1 in \cite{jacobson1990galois} and Theorem A in \cite{jacobson1990galois} the Galois group of Galois closure $\tilde{L}$ over $\mathbb{Q}$ is $G=\mathbb{Z}/8 \mathbb{Z} \rtimes (\mathbb{Z}/8 \mathbb{Z})^{\times}=\mathbb{Z}/8 \mathbb{Z} \rtimes (\mathbb{Z}/2 \mathbb{Z}\times \mathbb{Z}/2 \mathbb{Z})$ and $M\subset \tilde{L}$. Now consider $K$ to be any Hilbertian field. Now $(\Z/8\Z)$ is finite abelian. Also $\Z / 2\Z\times \Z/2\Z$ is regular over $K$ by Lemma \ref{hilb fie} (3) and it clearly acts on $(\Z/8\Z)$. Thus by Lemma \ref{hilb fie} (6), $G$ is realizable as a Galois group over $K$. By Lemma \ref{perm} we are done.

\medskip

\item 
Now we prove the theorem. We break the problem into cases.\smallskip 
 
 Case 1: $n=p^{\mu}$. Now $a=t\cdot p^{\mu}\mid n$. So $t=1$ and $a=p^{\mu}=n$. As $a\mid \alpha\mid  n$. So $\alpha=a$. Thus $\tau=t$. As $K$ is imperfect, by Lemma \ref{p insep} (1), we have a purely inseparable simple extension $L/K$ where $L=K(\alpha)$ for some $\alpha\in \bar{K}$ with $[L:K]=p^{\mu}=n$. Thus $t_K(L)=1$ and $a_K(L)=a$. As $n>1$, so for $M=K$ we have $\tau_K(M,L)=0=\alpha_K(M,L)$. For $M=L$ we have $\tau_K(M,L)=[L_S:K]=1$ and $\alpha_K(M,L)=[L:K]=p^{\mu}$.

 \smallskip
 
Case 2: $n>2p^{\mu}$.  Let $n'=n/p^{\mu}$. By the assumptions, we have $n'>2$ and $t\ |\ \tau\ |\ n'$. Also in the case when $t=1$ we have $\tau\neq 2$ and $n'\neq 2\tau$ and in the case when $t$ is odd we have either $\tau\neq 2t$ or $n'\neq 2\tau$. As $K$ is Hilbertian, so by part (1), we have separable extensions $S/K$ and $M'/K$ with $S\subset M'\subset \tilde{S}$ such that $[S:K]=n'$ and $t_K(S)=t$ and $\tau_K(M', S)=\tau$. As $K$ is imperfect, by Lemma \ref{p insep} (1), we have a purely inseparable simple extension $I/K$ with $[I:K]=p^{\mu}$.  Let $L=SI$ and $M=M'I$. So $L_S=S$ and $L_I=I$ and $L=L_SL_I$. By Proposition \ref{gen norm for L} (1) and Proposition \ref{GNM theorem} and Proposition \ref{GNM thm II}, $t_K(L)=t_K(L_S)$. So $t_K(L)=t$ and $a_K(L)=t_K(L)\cdot i_K(L)=t\cdot p^{\mu}=a$ and $\tau_K(M,L)=\tau_K(M',S)\cdot \tau_K(I,I)=\tau\cdot 1=\tau$ and $\alpha_K(M,L)=\alpha_K(M',S)\cdot \alpha_K(I,I)=\tau\cdot p^{\mu}=\alpha$.\smallskip

Case 3: $n=2p^{\mu}$. As $a\mid n$, so $t\cdot p^{\mu}\mid 2p^{\mu}$ and so $t \mid 2$. As we have assumed $t\neq 1$ when $n=2p^{\mu}$, we have $t=2$. So $\alpha=a=2p^{\mu}=n$ and thus $\tau=2$. By Lemma \ref{hilb fie} (3), $\mathbb{Z}/2\mathbb{Z}$ is realizable as a Galois group over $K$. Let $S/K$ be the degree $2$ Galois extension. So $t_K(S)=2$. As $K$ is imperfect, by Lemma \ref{p insep} (1), we have a purely inseparable simple extension $I/K$ with $[I:K]=p^{\mu}$.  Let $L=SI$. So $L_S=S$ and $L_I=I$ and $L=L_SL_I$ and $L/K$ is normal. Thus $[L:K]=2p^{\mu}$ and $t_K(L)=t_K(L_S)=2$. So $a_K(L)=t_K(L)\cdot i_K(L)=2\cdot p^{\mu}=n$. For $M=L$ we have $\tau_K(M,L)=2$ and $\alpha_K(M,L)=n$.

\smallskip

In Case 2 and Case 3, we observe that as $S/K$ is separable, so by primitive element theorem $S=K(\beta)$ for some $\beta\in \bar{K}$. As $I/K$ is a purely inseparable simple extension, $I=K(\gamma)$ for some $\gamma\in \bar{K}$. So $L=K(\beta, \gamma)$. By stronger version of primitive element theorem, Theorem 1 in \cite{brown2012mathematics} (or by Theorem C.1 in \cite{conradseparability}), $L=K(\alpha)$ for some $\alpha\in \bar{K}$.

\end{enumerate}
\end{proof}

\begin{proof}[Proof of Theorem \ref{gamma gen res}]

 We prove the theorem in two parts.

\begin{enumerate}
\item First we will show that: Given $(n,\gamma)$ where $n>2$ and $n \mid \gamma\mid  n!$. Assume that $n=2^m a_1a_2\cdots a_k$ with each $a_i>2$ and $m=0$ or $1$ and $\gamma=2^m b_1b_2\cdots b_k$ with each (i) $b_i=\ ^{a_i}P_j$ for $j\leq a_i$ or (ii) $b_i=a_i\phi(a_i/l)$ for $a_i$ odd and $l\mid a_i$ or (iii) $b_i=a_i\cdot r^{a-1}$ for $r>1$ and $r\mid a_i$ and $a\leq (a_i/r)$. There exist separable extensions $L/K$ and $M/K$ such that $[L:K]=n$ and compositum indicium $\gamma_K (M,L)=\gamma$. We get $M/K$ as an extension of $L/K$ contained in $\tilde{L}$.\smallskip

   Firstly we show existence of $L_i/K$ and $M_i/K$ with $M_i\subset \tilde{L_i}$ for $1\leq i\leq k$ such that $[L_i:K]=a_i$ and $\gamma_K(M_i,L_i)=b_i$.
   
 \smallskip  
   
Case (i): $b_i=\ ^{a_i}P_j$ for $j\leq a_i$. This follows from Example 2.8 in \cite{jaiswal2025rootcapacityintersectionindicium} and Corollary \ref{Sn ref} for $\mathfrak{S}_n$.

\smallskip

Case (ii): $b_i=a_i\phi(a_i/l)$ for $a_i$ odd and $l\mid a_i$. For $K=\mathbb{Q}$, this case follows from Proposition 2.22 in \cite{jaiswal2025rootcapacityintersectionindicium}. By Proposition 1 in \cite{jacobson1990galois} and Theorem A in \cite{jacobson1990galois}, the Galois group of Galois closure $\tilde{L_i}$ over $\mathbb{Q}$ is $\mathbb{Z}/a_i\mathbb{Z}\rtimes (\mathbb{Z}/a_i\mathbb{Z})^{\times}$ and $M_i\subset \tilde{L_i}$. Now consider $K$ to be any Hilbertian field. Now $(\Z/a_i\Z)$ is finite abelian. Also $(\Z / a_i\Z)^{\times}$ is regular over $K$ by Lemma \ref{hilb fie} (3) and it clearly acts on $(\Z/a_i\Z)$. Thus by Lemma \ref{hilb fie} (6), $G$ is realizable as a Galois group over $K$. By Lemma \ref{perm} we are done.

\smallskip

Case (iii) $b_i=a_i\cdot r^{a-1}$ for $r>1$ and $r\mid a_i$ and $a\leq (a_i/r)$. Consider the construction in first part of the proof of Theorem \ref{root cap gen res} for $r>1$. So we have $L_i/K$ and $M_i/K$ such that $[L_i:K]=a_i$ and $r_K(L_i)=r$ and $\rho_K(M_i,L_i)=\rho=a\cdot r$. For that $M_i/K$ we have by computation that $\gamma_K(M_i,L_i)=[M_i:K]=a_i\cdot r^{a-1}$.

\smallskip

Now by Lemma \ref{lin disj}, $L_i$'s can be chosen such that for each $1\leq t\leq k-1$ we have that $\tilde{L_1}\cdots \tilde{L_{t}}$ and $\tilde{L_{t+1}}$ are linearly disjoint over $K$. Let $L'=L_1\cdots L_k$ and $M'=M_1\cdots M_k$. Thus by Propositions \ref{GNM theorem} and \ref{GNM thm II}, it follows that $[L':K]=a_1\cdots a_k$ and $\gamma_K(M',L')=\gamma_K(M_1,L_1)\cdots\gamma_K(M_k,L_k)=b_1\cdots b_k$. Thus we are done for $m=0$. Now suppose $m=1$, By Lemma \ref{hilb fie} (3) and Lemma \ref{lin disj}, there exist a Galois extension $F/K$ of degree $2$ such that $\tilde{L'}$ and $F$ are linearly
disjoint over $K$. Let $L=L'F$ and $M=M'F$. Thus $[L:K]=2 a_1\cdots a_k = n$ and $\gamma_K(M,L)=\gamma_K(M',L')\cdot \gamma_K(F,F)=2b_1\cdots b_k=\gamma$.

\medskip

\item Now we prove the theorem. We break the problem into cases.\smallskip

Case 1: $n=p^{\mu}$. Then $\Gamma=p^{\mu}$ and $\gamma=1$. As $K$ is imperfect, by Lemma \ref{p insep} (1), we have a purely inseparable simple extension $L/K$ where $L=K(\alpha)$ for some $\alpha\in \bar{K}$ with $[L:K]=p^{\mu}=n$. As $n>1$, so for $M=K$ we have $\gamma_K(M,L)=0=\Gamma_K(M,L)$. For $M=L$ we have $\gamma_K(M,L)=[L_S:K]=1$ and $\Gamma_K(M,L)=[L:K]=p^{\mu}$. 

\smallskip

Case 2: $n>2p^{\mu}$.  Let $n'=n/p^{\mu}$. By the assumptions, we have $n'>2$. Also $n'=2^m a_1a_2\cdots a_k$ with each $a_i>2$ and $m=0$ or $1$ and $\gamma=2^m b_1b_2\cdots b_k$ with each (i) $b_i=\ ^{a_i}P_j$ for $j\leq a_i$ or (ii) $b_i=a_i\phi(a_i/l)$ for $a_i$ odd and $l\mid a_i$ or (iii) $b_i=a_i\cdot r^{a-1}$ for $r>1$ and $r\mid a_i$ and $a\leq (a_i/r)$. As $K$ is Hilbertian, so by part (1), we have separable extensions $S/K$ and $M'/K$ with $S\subset M'\subset \tilde{S}$ such that $[S:K]=n'$ and compositum indicium $\gamma_K (M',S)=\gamma$. As $K$ is imperfect, by Lemma \ref{p insep} (1), we have a purely inseparable simple extension $I/K$ with $[I:K]=p^{\mu}$.  Let $L=SI$ and $M=M'I$. So $L_S=S$ and $L_I=I$ and $L=L_SL_I$. By Proposition \ref{gen norm for L} (1) and Proposition \ref{GNM thm II}, $\gamma_K(M,L)=\gamma_K(M',S)\cdot \gamma_K(I,I)=\gamma\cdot 1=\gamma$ and $\Gamma_K(M,L)=\Gamma_K(M',S)\cdot \Gamma_K(I,I)=\gamma\cdot p^{\mu}=\Gamma$.

\smallskip

Case 3: $n=2p^{\mu}$. Then $\Gamma=2p^{\mu}$ and $\gamma=2$. By Lemma \ref{hilb fie} (3), $\mathbb{Z}/2\mathbb{Z}$ is realizable as a Galois group over $K$. Let $S/K$ be the degree $2$ Galois extension. As $K$ is imperfect, by Lemma \ref{p insep} (1), we have a purely inseparable simple extension $I/K$ with $[I:K]=p^{\mu}$.  Let $L=SI$. So $L_S=S$ and $L_I=I$ and $L=L_SL_I$ and $L/K$ is normal. Thus $[L:K]=2p^{\mu}$. For $M=L$ we have $\gamma_K(M,L)=2$ and $\Gamma_K(M,L)=n$.

\smallskip

We also have $\iota_K(M,L)=i_K(L)=\Gamma/\gamma$ in all the cases. In Case 2 and Case 3, we observe that as $S/K$ is separable, so by primitive element theorem $S=K(\beta)$ for some $\beta\in \bar{K}$. As $I/K$ is a purely inseparable simple extension, $I=K(\gamma)$ for some $\gamma\in \bar{K}$. So $L=K(\beta, \gamma)$. By stronger version of primitive element theorem, Theorem 1 in \cite{brown2012mathematics} (or by Theorem C.1 in \cite{conradseparability}), $L=K(\alpha)$ for some $\alpha\in \bar{K}$.

\end{enumerate}\end{proof}

    

\bigskip

\noindent {\it Acknowledgements:} The author would like to thank his Post Doc Mentor Prof Tony Puthenpurakal, IIT Bombay for suggesting to work on the theory of root clusters for global fields with non zero characteristic which resulted in this paper. The author is also thankful to Prof Dipendra Prasad, Prof Sudhir Ghorpade and Prof P Vanchinathan for interesting discussions related to this topic. The author would also like to acknowledge support of IIT Bombay Institute Post Doctoral Fellowship during this work.

\medskip

\bibliographystyle{plain}
 \bibliography{mybib}

\end{document}